\documentclass[11pt]{article}%
\usepackage{eurosym}
\usepackage{amssymb}
\usepackage{amsmath}
\usepackage{amsfonts}
\usepackage{boxedminipage}
\usepackage{color}
\usepackage{fancyhdr}
\usepackage{graphicx}
\usepackage{hyperref}%
\providecommand{\U}[1]{\protect\rule{.1in}{.1in}}
\newtheorem{proposition}{Proposition}[section]
\newtheorem{theorem}[proposition]{Theorem}

\newtheorem{lemma}[proposition]{Lemma}

\newtheorem{definition}[proposition]{Definition}
\newtheorem{remark}[proposition]{Remark}
\newtheorem{example}[proposition]{Example}

\newtheorem{condition}[proposition]{Condition}

\numberwithin{equation}{section}
\numberwithin{proposition}{section}
\newenvironment{proof}[1][Proof]{\noindent\textbf{#1.} }{\ \rule{0.5em}{0.5em}}

\begin{document}

\title{Local stability of Kolmogorov forward equations for finite state nonlinear
Markov processes}
\author{Amarjit Budhiraja\thanks{Research supported in part by National Science Foundation(DMS-1305120) and the Army Research Office (W911NF-10-1-0158, W911NF- 14-1-0331)}, Paul Dupuis\thanks{Research supported in part by the Army
Research Office (W911NF-12-1-0222).}, Markus Fischer and Kavita Ramanan\thanks{Research supported in part by the Army Research Office (W911NF-12-1-0222)
and the National Science Foundation (NSF CMMI-1234100 and NSF DMS-1407504)}}
\maketitle

\begin{abstract}
The focus of this work is on local stability of a class of nonlinear ordinary
differential equations (ODE) that describe limits of empirical measures
associated with finite-state exchangeable weakly interacting $N$-particle 
systems. Local Lyapunov functions are identified for several classes of such
ODE, including those associated with systems with slow adaptation and Gibbs
systems. Using results from \cite{BuDuFiRa1} and large deviations heuristics,
a partial differential equation (PDE) associated with the nonlinear ODE is
introduced and it is shown that positive definite subsolutions of this PDE
serve as local Lyapunov functions for the ODE. This PDE characterization is
used to construct explicit Lyapunov functions for a broad class of models
called locally Gibbs systems. This class of models is significantly larger
than the family of Gibbs systems and several examples of such systems are
presented, including models with nearest neighbor jumps and models with
simultaneous jumps that arise in applications.

\end{abstract}

%\tableofcontents

%\bigskip \medskip The purpose of these notes is to explore connections
%between the large deviation properties of collections of weakly interacting,
%finite-state Markov chains, and the construction of Lyapunov functions. In
%the next section we establish notation for the model when only one chain can
%jump at a time.



\section{Introduction}

In this paper we consider local stability properties of the nonlinear ordinary
differential equation (ODE)
\begin{equation}
\frac{d}{dt}p(t)=p(t)\Gamma(p(t)). \label{EqLimitKolmogorov}%
\end{equation}
where $p(t)$ takes values in ${\mathcal P}({\mathcal X})$.
Here $\mathcal{X}$ is a finite set that we denote by ${\mathcal{X}}%
=\{1,\ldots,d\}$, ${\mathcal{P}}({\mathcal{X}})$ is the space of probability
measures on $\mathcal{X}$ equipped with the topology of weak convergence,
which we identify with the unit $(d-1)$-dimensional simplex ${\mathcal{S}%
}=\{r\in\mathbb{R}^{d}:r_{x}\geq0,x\in{\mathcal{X}},\mbox{ and }\sum
_{x\in{\mathcal{X}}}r_{x}=1\}$ and for each $p\in{\mathcal{P}}({\mathcal{X}}%
)$, $\Gamma(p)$ is a rate matrix for a Markov chain on ${\mathcal{X}}$. Such
ODEs describe the evolution of the law of so-called nonlinear Markov or
McKean-Vlasov processes that arise as limits of weakly interacting Markov
chains (see for example Section 2 of the companion paper \cite{BuDuFiRa1}). In
this context, the ODE \eqref{EqLimitKolmogorov} is referred to as the forward
equation of the nonlinear Markov process. The focus of the current paper is
local stability (see Definition \ref{def:locstab}) of the ODE
\eqref{EqLimitKolmogorov}, and therefore of the corresponding nonlinear Markov
process, for several families of models.

As usual in the study of stability of dynamical systems, the basic approach is
to construct a suitable local Lyapunov function (see Definition \ref{loclyap}%
). It is known (see, for example, Section 3 of \cite{BuDuFiRa1}) that for an
ergodic \textit{linear} Markov process on $\mathcal{X}$ (i.e., the case where $\Gamma$ is
constant) the mapping $q\mapsto R(q\Vert\pi)$, where $R$ is relative entropy
and $\pi$ is the unique stationary distribution, defines a Lyapunov function
for the associated linear Kolmogorov equation. Although one does not expect
this property to hold for general nonlinear Markov processes (see Section 3 of
\cite{BuDuFiRa1} for a discussion of this point), in Section
\ref{SectSlowAdaptation} we consider a family of models, which we call systems
with slow adaptation, for which relative entropy is in fact a Lyapunov
function when the adaptation parameter is sufficiently small. This result says
that relative entropy continues to serve as a Lyapunov function for suitably
small non-linear perturbations of linear Markov processes, but it does not
yield Lyapunov functions for general nonlinear Markov processes. For one
particular family of models whose stationary distributions take an explicit form and which we call systems of Gibbs type, Section 4 of
\cite{BuDuFiRa1} proposed a local Lyapunov function defined as the limit of
certain scaled relative entropies that involve the stationary distributions of
the associated $N$-particle weakly interacting Markov processes. In Section
\ref{SectGibbs} of the current work we show that this function is in fact a
local Lyapunov function in the sense of Definition \ref{loclyap} under
suitable positive definiteness assumptions.

For non-Gibbs families, stationary distributions usually will not take an
explicit form and thus a different approach is needed. One such approach was
developed in Section 5 of \cite{BuDuFiRa1}, where  analogous limits of
scaled relative entropies, but with the stationary distributions of the
$N$-particle system replaced by the joint law of the $N$-particles at time
$t$, were identified in terms of the large deviation rate function $J_{t}%
(\cdot)$ for the empirical measure of the state of the weakly interacting
Markov process at time $t$. The limit of $J_{t}$ as $t\rightarrow\infty$ was
proposed in \cite{BuDuFiRa1} as a local Lyapunov function for the ODE
\eqref{EqLimitKolmogorov}, though the question of when these limits exist and
how they can  be evaluated was not tackled. In this work we approach this
question as follows. We begin by formally deriving a nonlinear partial
differential equation (PDE) for $\{J_{t}(q),t\geq0,q\in\mathcal{S}\}$. We next show
that classical sense positive definite subsolutions of the stationary form of
the PDE (see \eqref{eq:statpde}), which is formally the equation governing the limit of $J_t$ as $t\to \infty$, are local Lyapunov functions for
\eqref{EqLimitKolmogorov}. With this result, the problem of constructing
Lyapunov functions reduces to finding suitable subsolutions of
\eqref{eq:statpde}. Although finding explicit subsolutions can be challenging
in general, in Section \ref{sec:examples} we introduce an interesting family
of models, which we call locally Gibbs systems, for which one can in fact give
an explicit solution for \eqref{eq:statpde}. These models contain, as a
special case, the Gibbs type systems studied in Section \ref{SectGibbs}.
Moreover, in Sections \ref{subs-gGibbs} -- \ref{subs-lGibbs2} we present other
examples of locally Gibbs systems, including models with nearest neighbor
jumps and models with simultaneous jumps that arise in telecommunications
applications. Finally we give an example to illustrate that solutions to the
PDE \eqref{eq:statpde} can be found for systems that are not locally Gibbs as well.

%The starting point for our local stability analysis will be to
%show that certain functions that were obtained in \cite{BuDuFiRa1} as limits of relative entropies, will
%serve as local Lyapunov functions(in the sense of Definition \ref{loclyap}) for the ODE.
%In addition, guided by the relative entropy asymptotics of Section 6
%in \cite{BuDuFiRa1}, we introduce a nonlinear
%partial differential equation (PDE) \eqref{eq:statpde} which can also be used
%to identify Lyapunov functions. Specifically, in Section \ref{sec-PDE} we
%show that positive definite (cf. Definition \ref{posdef})
%subsolutions of the PDE \eqref{eq:statpde} are local Lyapunov functions for the
%dynamical system \eqref{EqLimitKolmogorov}. We also identify a broad family of
%models that we refer to as \textquotedblleft locally Gibbs
%systems\textquotedblright\ for which a local Lyapunov function can be
%constructed using this PDE characterization. Other works that have studied the
%long-time behavior of weakly interacting Markov jump processes arising in
%applications include
%\cite{AntFriRobTib08,BorMcdPro12,GomGraLeB12,GraRob10,Tib10}, where Lyapunov
%functions have been constructed in special cases. Several
%of these examples are special cases of locally Gibbs systems and we discuss one of these
%in some detail in Section \ref{sec:examples}.

The paper is organized as follows. Section \ref{subs-def} collects some
definitions and basic results related to stability of the ODE
\eqref{EqLimitKolmogorov}. In Section \ref{SectSlowAdaptation} we study
systems with slow adaptation. Section \ref{SectGibbs} considers the setting of
systems of Gibbs type. We then study more general models than the Gibbs
systems of Section \ref{SectGibbs}. In Section \ref{subs-timedeppde}, we
present the formal derivation of a nonlinear time-dependent PDE that is
satisfied by the large deviation rate function $\{J_{t}(q)\}$.
%  obtained in
% \cite{DupRamWu12} and introduced in Theorem 5.4 of \cite{BuDuFiRa1}. 
The main
result of this section shows that a positive definite subsolution of the
stationary version of this PDE is a local Lyapunov function of
\eqref{EqLimitKolmogorov}. Finally, in Section \ref{sec:examples} we identify
a broad family of models, referred to as locally Gibbs systems, for which a
non-trivial subsolution of \eqref{eq:statpde} can be given explicitly and thus
under suitable additional conditions that ensure positive definiteness, one
can obtain tractable Lyapunov functions for such systems, ensuring local
stability. We also present several examples that illustrate the range of
applicability of these results.

\section{Local Stability and Lyapunov Functions}

\label{subs-def} In this section we will collect some definitions and basic
results related to stability of the dynamical system
\eqref{EqLimitKolmogorov}. The following condition will be assumed on several occasions.
\begin{condition}
	\label{cond:lip}
	The function $p\mapsto
	\Gamma(p)$ is a Lipschitz continuous map from $\mathcal{S}$ to $\mathbb{R}$.
\end{condition}  
Some results,
such as the main result of this section (Proposition \ref{stab-loc}), only need
 that $\Gamma$ be continuous, which is sufficient to ensure the existence
of a solution for any initial condition. Denote by ${\mathcal{S}}^{\circ}$ the
relative interior of ${\mathcal{S}}$:
\[
\mathcal{S}^{\circ}\doteq\{p\in\mathcal{S}:p_{i}>0\mbox{ for all }i=1,\ldots
,d\}.
\]
We first recall the definition of a locally stable fixed point of an ODE.

\begin{definition}
\label{def-fp} A point $\pi^{*} \in{\mathcal{S}}$ is said to be a
\textbf{fixed point} of the ODE (\ref{EqLimitKolmogorov}) if the right-hand
side of (\ref{EqLimitKolmogorov}) evaluated at $p = \pi^{*}$ is equal to zero,
namely,
\[
\pi^{*} \Gamma(\pi^{*}) = 0.
\]

\end{definition}

\begin{definition}
\label{def:locstab} A fixed point $\pi^{*} \in{\mathcal{S}}^{\circ}$ of the
ODE (\ref{EqLimitKolmogorov}) is said to be \textbf{locally stable} if there
exists a relatively open subset $\mathbb{D}$ of ${\mathcal{S}}$ that contains
$\pi^{*}$ and has the property that whenever $p(0) \in\mathbb{D}$, the
solution $p(t)$ of (\ref{EqLimitKolmogorov}) with initial condition $p(0)$
converges to $\pi^{*}$ as $t \to\infty$.
\end{definition}

Our approach to proving local stability will be based on the construction of
suitable Lyapunov functions. In order to state the Lyapunov function property
precisely, we begin with some notation. Let
\[
\mathcal{H}_{1}\doteq\left\{  v\in\mathbb{R}^{d}:\sum_{i=1}^{d}v_{i}%
=1\right\}
\]
be the hyperplane containing the simplex ${\mathcal{S}}$, and let
\[
\mathcal{H}_{0}\doteq\left\{  v\in\mathbb{R}^{d}:\sum_{i=1}^{d}v_{i}%
=0\right\}
\]
be a shifted version of this hyperplane that goes through the origin.

Given a set $\mathbb{D}\subset\mathcal{H}_{1}$, a function $U:\mathbb{D}%
\rightarrow\mathbb{R}$ will be called \emph{differentiable} (respectively
${\mathcal{C}}^{1}$) if it is differentiable (respectively, continuously
differentiable) on some relatively open subset $\mathbb{D}^{\prime}$ of
${\mathcal{H}}_{1}$ such that $\mathbb{D}\subset\mathbb{D}^{\prime}$. In
particular, for a differentiable function $U$ on a relatively open subset
$\mathbb{D}$ of ${\mathcal{H}}_{1}$, for every $r\in\mathbb{D}$, there exists
a unique vector $D^{\mbox{\tiny{tan}}}U(r)\in\mathcal{H}_{0}$, called the
gradient of $U$ at $r$, such that
\[
\lim_{h\in\mathcal{H}_{0},\,\Vert h\Vert\rightarrow0}\frac{U(r+h)-U(r)-\langle
D^{\mbox{\tiny{tan}}}U(r),h\rangle}{\Vert h\Vert}=0.
\]
Note that if $\{h_{i},i=1,\ldots,d-1\}$ is an orthonormal basis of the
subspace $\mathcal{H}_{0}$, we can write
\[
D^{\mbox{\tiny{tan}}}U(r)=\sum_{i=1}^{d-1}\langle D^{\mbox{\tiny{tan}}}%
U(r),h_{i}\rangle h_{i},\;\quad r\in\mathbb{D}.
\]
Finally, we say that the differentiable function $U:\mathbb{D}\rightarrow
\mathbb{R}$ is ${\mathcal{C}}^{1}$ if the mapping $r\mapsto
D^{\mbox{\tiny{tan}}}U(r)$ from $\mathbb{D}$ to $\mathcal{H}_{0}$ is
continuous. Frequently, with an abuse of notation, we write
$D^{{\mbox{\tiny{tan}}}}U$ simply as $DU$.

We introduce the following notion of \emph{positive definiteness.}

\begin{definition}
\label{posdef} Let $\pi^{\ast}\in{\mathcal{S}}^{\circ}$ be a fixed point of
\eqref{EqLimitKolmogorov} and let $\mathbb{D}$ be a relatively open subset of
${\mathcal{S}}$ that contains $\pi^{\ast}$. A function $J:\mathbb{D}%
\rightarrow\mathbb{R}$ is called \textbf{positive definite} if for some
$K^{\ast}\in\mathbb{R}$, the sets $M_{K}=\{r\in\bar{\mathbb{D}}:J(r)\leq K\}$
decrease continuously to $\{\pi^{\ast}\}$ as $K\downarrow K^{\ast}$.
\end{definition}

In Definition \ref{posdef}, by \textquotedblleft decrease continuously to
$\{\pi^{\ast}\}$\textquotedblright\ we mean that: (i) for every $\epsilon>0$,
there exists $K_{\epsilon}\in(K^{\ast},\infty)$ such that $M_{K_{\epsilon}%
}\subset\mathbb{B}_{\epsilon}(\pi^{\ast})\cap\mathbb{D}$, where $\mathbb{B}%
_{\epsilon}(\pi^{\ast})$ is the open Euclidean ball of radius $\epsilon$, centered
at $\pi^{\ast}$, and (ii) for every $K>K^{\ast}$, there exists $\epsilon>0$
such that $\mathbb{B}_{\epsilon}(\pi^{\ast})\cap\mathcal{S}\subset M_{K}$.
Note that if $J$ is a uniformly continuous function on $\mathbb{D}$ which
attains its minimum uniquely at $\pi^{\ast}$ then
%such that $J(\pi
%^{*})=0$, $J(p) > 0$ for $p \neq\pi$, and $J(p_{n}) \to0$ implies $p_{n}
%\to\pi^{*}$ for any sequence $\{p_{n}\} \subset\mathbb{D}$, then
$J$ is positive definite. A basic example of such a function is the relative
entropy function $p\mapsto R(p\Vert\pi^{\ast})$ introduced in the next section.

\begin{definition}
\label{loclyap} Let $\pi^{\ast}\in{\mathcal{S}}^{\circ}$ be a fixed point of
\eqref{EqLimitKolmogorov}, and let $\mathbb{D}$ be a relatively open subset of
${\mathcal{S}}$ that contains $\pi^{\ast}$. A positive definite,
${\mathcal{C}}^{1}$ and uniformly continuous function $J:\mathbb{D}%
\rightarrow\mathbb{R}$ is said to be a \textbf{local Lyapunov function}
associated with $(\mathbb{D},\pi^{\ast})$ for the ODE
\eqref{EqLimitKolmogorov} if, given any $p(0)\in\mathbb{D}$, the solution
$p(\cdot)$ to the ODE \eqref{EqLimitKolmogorov} with initial condition $p(0)$
satisfies $\frac{d}{dt}J(p(t))<0$ for all $0\leq t<\tau$ such that
$p(t)\neq\pi^{\ast}$, where $\tau\doteq\inf\{t\geq0:p(t)\in\mathbb{D}^{c}\}$.
In the case $\mathbb{D}=\mathcal{S}^{\circ}$, we refer to $J$ as a
\textbf{Lyapunov function}.
\end{definition}

The following result shows that, as one would expect, existence of a local
Lyapunov function implies local stability. The proof is standard, but is
included for completeness.

\begin{proposition}
\label{stab-loc} Let $\pi^{\ast}\in{\mathcal{S}}^{\circ}$ be a fixed point of
\eqref{EqLimitKolmogorov} and suppose  that Condition \ref{cond:lip}
holds. Suppose there exists a local Lyapunov function associated
with $(\mathbb{D},\pi^{\ast})$ for \eqref{EqLimitKolmogorov} where
$\mathbb{D}$ is some relatively open subset of $\mathcal{S}$ that contains
$\pi^{\ast}$. Then $\pi^{\ast}$ is locally stable.
\end{proposition}

\begin{proof}
Let $J$ be a local Lyapunov function associated with $(\mathbb{D}, \pi^{*})$
for \eqref{EqLimitKolmogorov}. Since $J$ is positive definite, there exists
$K^{*} \in\mathbb{R}$ such that the sets $M_{K}=\{r\in\bar{\mathbb{D}%
}:J(r)\leq K\}$ decrease continuously to $\{\pi^{\ast}\}$ as $K\downarrow
K^{\ast}$. In particular, there exists $L \in(K^{*}, \infty)$ and a relatively
open subset $\mathbb{D}_{0}$ of $\mathcal{S}$ such that $\pi^{*} \in\mathbb{D}_{0}
\subset M_{L} \subset\mathbb{D}$.

We will prove that \eqref{EqLimitKolmogorov} is locally stable on
$\mathbb{D}_{0}$, namely
\begin{equation}
\mbox{ whenever }p(0)\in\mathbb{D}_{0}%
,\mbox{ the solution }p(t)\mbox{ of 	\eqref{EqLimitKolmogorov} converges to }\pi
^{\ast}\mbox{ as }t\rightarrow\infty. \label{eq:1715}%
\end{equation}
Note that \eqref{eq:1715} is clearly true if $p(0)=\pi^{\ast}$. Suppose now
that $p(0)\neq\pi^{\ast}$. If $p(t)\in\mathbb{D}$ then
\[
\frac{d}{dt}J(p(t))=\langle DJ(p(t)),p(t)\Gamma(p(t))\rangle.
\]
Let $\tau\doteq\inf\{t\geq0:p(t)\in\mathbb{D}^{c}\}$, and assume that
$\tau<\infty$. Since $q\mapsto\langle DJ(q),q\Gamma(q)\rangle$ is a continuous
function on $\mathbb{D}$ and $\frac{d}{dt}J(p(t))<0$ whenever $p(t)\in
\mathbb{D}\setminus\{\pi^{\ast}\}$, we have  $\frac{d}{dt}J(p(t))\leq0$
for all $t\in\lbrack0,\tau)$. Combining this with the fact that $J$ extends
continuously to $\bar{\mathbb{D}}$ we have  $J(p(\tau))\leq L$ and
consequently $p(\tau)\in M_{L}\subset\mathbb{D}$. This contradicts the
assumption $\tau<\infty$. Hence $\tau=\infty$ and
\begin{equation}
\frac{d}{dt}J(p(t))<0\mbox{ for all }t\geq0\mbox{ whenever }p(t)\neq\pi^{\ast
}. \label{eq:negalways}%
\end{equation}

Let $K_{n}\in(K^{\ast},L)$ be a strictly decreasing sequence such that
$K_{n}\downarrow K^{\ast}$ as $n\rightarrow\infty$. Let
\[
\tau_{n}=\inf\{t\geq0:p(t)\in M_{K_{n}}\}.
\]
Note that if $\tau_{n}<\infty$, then $p(t)\in M_{K_{n}}$ for all $t\geq
\tau_{n}$. Since the sets $M_{K_{n}}$ decrease continuously to $\{\pi^{\ast}\}$, it
suffices to show that $\tau_{n}<\infty$ for every $n$.

Consider $n=1$. If $p(0)\in M_{K_{1}}$, $\tau_{1}<\infty$ is immediate.
Suppose now that $p(0)\not \in M_{K_{1}}$. Let $\varepsilon_{0}>0$ be such
that $\mathbb{B}_{\varepsilon_{0}}(\pi^{\ast})\cap\mathcal{S}\subset M_{K_{1}%
}\subset M_{L}$. From \eqref{eq:negalways}, for every $q\in(\mathbb{B}%
_{\varepsilon_{0}})^{c}\cap M_{L}$, $\frac{d}{dt}J(p^{q}(t))|_{t=0}<0$ where
$p^{q}(t)$ is the solution of \eqref{EqLimitKolmogorov} with $p(0)=q$.
Recalling the continuity of $q\mapsto\langle DJ(q),q\Gamma(q)\rangle$ and
observing that $(\mathbb{B}_{\varepsilon_{0}})^{c}\cap M_{L}$ is a closed
subset of $\mathbb{D}$ we have that
\[
\sup_{q\in(\mathbb{B}_{\varepsilon_{0}})^{c}\cap M_{L}}\langle DJ(q),q\Gamma
(q)\rangle<0.
\]
Also, since $\mathbb{B}_{\varepsilon_{0}}(\pi^{\ast})\cap\mathcal{S}\subset
M_{K_{1}}$, for all $t<\tau_{1}$, $p(t)\in(\mathbb{B}_{\varepsilon_{0}}%
)^{c}\cap M_{L}$. Thus we have that $\sup_{t<\tau_{1}}\frac{d}{dt}J(p(t))<0$.
This shows that $\tau_{1}<\infty$. By repeating this argument we see that
$\tau_{n}<\infty$ for every $n$, and the result follows.
\end{proof}

\section{Systems with Slow Adaptation}

\label{SectSlowAdaptation}

Here we consider the case where the ODE \eqref{EqLimitKolmogorov} exhibits a
structure we call \emph{slow adaptation}, for which the strength of the
nonlinear component is adjusted through a small parameter. The long-time
behavior of systems of this type, in the context of nonlinear diffusions
arising as limits of weakly interacting It{\^{o}} diffusions, is studied in
\cite{veretennikov06} based on coupling arguments and hitting times (and not
in terms of Lyapunov functions).

Suppose that Condition \ref{cond:lip} holds and $\pi^{\ast}\in\mathcal{P} (\mathcal{X})$ is a fixed point of the
ODE \eqref{EqLimitKolmogorov}. The rate matrix $\Gamma^{\lambda}(p)=$
$\Gamma(\lambda(p-\pi^{\ast})+\pi^{\ast})$ corresponds to a version of the
original system but with slow adaptation\ when $\lambda>0$ is small. With
$\lambda\in(0,1]$ fixed, the rate matrices $\Gamma^{\lambda}(p)$,
$p\in\mathcal{P}(\mathcal{X})$, determine a family of nonlinear Markov
processes. The corresponding forward equation
\begin{equation}
\frac{d}{dt}p^{\lambda}(t)=p^{\lambda}(t)\Gamma^{\lambda}(p^{\lambda}(t))
\label{EqSlowKolmogorov}%
\end{equation}
has a unique solution given any initial distribution $p(0)\in\mathcal{P}%
(\mathcal{X})$. Note that for any $\lambda\in\lbrack0,1]$, $\pi^{\ast}$ is
also a fixed point for \eqref{EqSlowKolmogorov}. We are interested in the
question of when the fixed point $\pi^{\ast}$ is locally stable for
sufficiently slow adaptation.

Recall that given $p,\pi^{\ast}\in{\mathcal{P}}({\mathcal{X}})$, the relative
entropy of $p$ with respect to $\pi^{\ast}$ is given by
\begin{equation}
R\left(  p\Vert\pi^{\ast}\right)  \doteq\sum_{x\in{\mathcal{X}}}p_{x}%
\log\left(  \displaystyle\frac{p_{x}}{\pi_{x}^{\ast}}\right)  .
\label{def-Relent}%
\end{equation}
It is known (see, e.g., \cite[pp. I-16-17]{Spi71} or \cite[Lemma
3.1]{BuDuFiRa1}) that the mapping
\begin{equation}
\bar{F}(p)=R\left(  p\Vert\pi^{\ast}\right)  , \label{ExSlowLyapunov}%
\end{equation}
serves as a Lyapunov function for finite-state linear Markov processes. The
forward equation of a finite-state linear Markov process has the form
\eqref{EqLimitKolmogorov}, but with a constant rate matrix $\Gamma$, and the
proof of the Lyapunov function property of relative entropy for such Markov
processes crucially uses the fact that $\Gamma$ is constant. In contrast,
since in general the rate matrix in the ODE \eqref{EqLimitKolmogorov} depends
on the state, one does not expect $R(\cdot\Vert\pi^{\ast})$ to serve as a
Lyapunov function for general finite-state nonlinear Markov processes.
Nevertheless, in this section we will show that for systems with slow
adaptation with $\lambda$ sufficiently small, the function $R(\cdot\Vert
\pi^{\ast})$ does in fact have the desired property.
The following is the main result of the section. Note that the function
$\bar{F}$ in \eqref{ExSlowLyapunov} is positive definite (in the sense of
Definition \ref{posdef}). Thus the Proposition below, together with Definition
\ref{loclyap}, says that $\bar{F}$ is a Lyapunov function associated with
$\pi^{\ast}$ for the ODE \eqref{EqSlowKolmogorov}.

\begin{proposition}
\label{PropSlow} Suppose Condition \ref{cond:lip} holds. Let $p^{\lambda}(\cdot)$ be defined by
(\ref{EqSlowKolmogorov}) and $\bar{F}$ by (\ref{ExSlowLyapunov}). Suppose that
$\Gamma(\pi^{\ast})$ is irreducible. Then there is $\lambda_{0}>0$ such that if $\lambda\in
\lbrack0,\lambda_{0}]$, then for all $t\geq0$%
\[
\frac{d}{dt}\bar{F}(p^{\lambda}(t))\leq0,
\]
with a strict inequality if and only if $p^{\lambda}(t)\neq\pi^{\ast}$.
\end{proposition}

\vspace{0pt} \noindent\textit{Proof.} By construction and hypothesis, there
exists $C\in(1,\infty)$ such that for all $x,y\in\mathcal{X}$, all $\lambda
>0$, and all $p\in\mathcal{P}(\mathcal{X})$,
\[
\bigl|\Gamma_{yx}^{\lambda}(p)-\Gamma_{yx}(\pi^{\ast})\bigr|\leq\lambda C\Vert
p-\pi^{\ast}\Vert,
\]
where $\Vert p-\pi^{\ast}\Vert\doteq\sum_{x\in\mathcal{X}}|p_{x}-\pi_{x}%
^{\ast}|$. Recall that since $\pi^{\ast}$ is stationary $\pi^{\ast}%
\Gamma^{\lambda}(\pi^{\ast})=0$. Using the definition (\ref{def-Relent}) of
relative entropy, the ODE (\ref{EqSlowKolmogorov}), and the relation
$\sum_{x,y\in{\mathcal{X}}}p_{y}\Gamma_{yx}^{\lambda}(p)=\sum_{x,y\in
{\mathcal{X}}}p_{y}\Gamma_{yx}(\pi^{\ast})=0$,
\begin{align*}
\frac{d}{dt}R\bigl(p^{\lambda}(t)\Vert\pi^{\ast}\bigr) &  =\sum_{x,y\in
\mathcal{X}}p_{y}^{\lambda}(t)\left(  \log\left(  \frac{p_{x}^{\lambda}%
(t)}{\pi_{x}^{\ast}}\right)  +1\right)  \Gamma_{yx}^{\lambda}(p^{\lambda
}(t))\\
&  =\sum_{x,y\in\mathcal{X}:x\neq y}p_{y}^{\lambda}(t)\left(  \log\left(
\frac{p_{x}^{\lambda}(t)\pi_{y}^{\ast}}{p_{y}^{\lambda}(t)\pi_{x}^{\ast}%
}\right)  -\frac{p_{x}^{\lambda}(t)\pi_{y}^{\ast}}{p_{y}^{\lambda}(t)\pi
_{x}^{\ast}}+1\right)  \Gamma_{yx}(\pi^{\ast})\\
&  \qquad+\sum_{x,y\in\mathcal{X}}p_{y}^{\lambda}(t)\log\left(  \frac
{p_{x}^{\lambda}(t)}{\pi_{x}^{\ast}}\right)  \left(  \Gamma_{yx}^{\lambda
}(p^{\lambda}(t))-\Gamma_{yx}(\pi^{\ast})\right)  \\
&  =\sum_{x,y\in\mathcal{X}:x\neq y}p_{y}^{\lambda}(t)\left(  \log\left(
\frac{p_{x}^{\lambda}(t)\pi_{y}^{\ast}}{p_{y}^{\lambda}(t)\pi_{x}^{\ast}%
}\right)  -\frac{p_{x}^{\lambda}(t)\pi_{y}^{\ast}}{p_{y}^{\lambda}(t)\pi
_{x}^{\ast}}+1\right)  \Gamma_{yx}(\pi^{\ast})\\
&  \qquad+\!\sum_{x,y\in\mathcal{X}:x\neq y}p_{y}^{\lambda}(t)\log\left(
\frac{p_{x}^{\lambda}(t)\pi_{y}^{\ast}}{p_{y}^{\lambda}(t)\pi_{x}^{\ast}%
}\right)  \left(  \Gamma_{yx}^{\lambda}(p^{\lambda}(t))-\Gamma_{yx}(\pi^{\ast
})\right)  ,
\end{align*}
where we use the convention that $0\log0=0$. For $x,y\in\mathcal{X}$ with
$x\neq y$ and $p\in{\mathcal{P}}({\mathcal{X}})$, set
\begin{align*}
\gamma_{yx}(p) &  \doteq p_{y}\left(  \log\left(  \frac{p_{x}\pi_{y}^{\ast}%
}{p_{y}\pi_{x}^{\ast}}\right)  -\frac{p_{x}\pi_{y}^{\ast}}{p_{y}\pi_{x}^{\ast
}}+1\right)  \Gamma_{yx}(\pi^{\ast}),\\
\rho_{yx}^{\lambda}(p) &  \doteq p_{y}\log\left(  \frac{p_{x}\pi_{y}^{\ast}%
}{p_{y}\pi_{x}^{\ast}}\right)  \left(  \Gamma_{yx}^{\lambda}(p)-\Gamma
_{yx}(\pi^{\ast})\right)  .
\end{align*}
To complete the proof we will show that there is $\lambda_{0}>0$ such that for
every $p\in{\mathcal{P}}({\mathcal{X}})$,
\begin{equation}
\sum_{x,y\in\mathcal{X}:x\neq y}\left(  \gamma_{yx}(p)+\rho_{yx}^{\lambda
}(p)\right)  \leq0\quad\text{for all }\lambda\in\lbrack0,\lambda
_{0}],\label{EqProofSlow}%
\end{equation}
with equality if and only if $p=\pi^{\ast}$.

% We first use the fact that $s\mapsto\log(s)-s+1$ is smooth and concave for
% $s\in(0,\infty)$, with the maximum value of zero at $s=1$. Thus $\sum
% _{x,y\in\mathcal{X}:x\neq y}\gamma_{yx}(p)\leq0$. If equality holds and
% $p_{y}>0$, then $p_{x}/p_{y}=\pi_{x}^{\ast}/\pi_{y}^{\ast}$ for all $x$ such
% that $\Gamma_{yx}(\pi^{\ast})>0$. Since $\Gamma(\pi^{\ast})$ is ergodic all
% states communicate, and hence $p=\pi^{\ast}$. 
It is straightforward to check
that $p\mapsto\gamma_{yx}(p)$ is concave. However we will need more than that, namely a uniform estimate on its second derivative. 
Let $r\in\mathcal{H}_{0}$ with
$\left\Vert r\right\Vert =1$. Evaluation of the derivatives gives
$$\left. \frac{d}{ds}\sum_{x,y\in\mathcal{X}:x\neq y}\gamma_{yx}%
(\pi^{\ast}+sr)\right\vert _{s=0}= 0,$$
\begin{equation}
\left.  \frac{d^{2}}{ds^{2}}\sum_{x,y\in\mathcal{X}:x\neq y}\gamma_{yx}%
(\pi^{\ast}+sr)\right\vert _{s=0}=-\sum_{x,y\in\mathcal{X}:x\neq y}\pi
_{y}^{\ast}\left(  \frac{r_{y}}{\pi_{y}^{\ast}}-\frac{r_{x}}{\pi_{x}^{\ast}%
}\right)  ^{2}\Gamma_{yx}(\pi^{\ast}).\label{eqn:2ndderiv}%
\end{equation}
If the expression in (\ref{eqn:2ndderiv}) is zero then, since $\pi_{y}^{\ast
}>0$ for all $y$ and all states communicate,%
\[
\frac{r_{y}}{\pi_{y}^{\ast}}=\frac{r_{x}}{\pi_{x}^{\ast}}%
\]
for all $x\neq y$. However this is impossible, since $r\in\mathcal{H}_{0}$
requires that at least one component be of the opposite sign of some other
component. Hence the expression in (\ref{eqn:2ndderiv}) is negative. Using
that $\{r:\left\Vert r\right\Vert =1\}$ is compact and continuity in $r$ show
that (\ref{eqn:2ndderiv}) is in fact bounded above away from zero on this set,
which shows the matrix of second derivatives is negative definite. Using the fact 
 that
$\mathcal{P}(\mathcal{X})$ is compact, we find that there is $c>0$, not
depending on $p$, such that
\[
\sum_{x,y\in\mathcal{X}:x\neq y}\gamma_{yx}(p)\leq-c\Vert p-\pi^{\ast}%
\Vert^{2},
\]
which is equivalent to
\begin{equation}
\sum_{x,y\in\mathcal{X}:x\neq y}\gamma_{yx}(p)\leq-\frac{c}{2}\Vert
p-\pi^{\ast}\Vert^{2}+\frac{1}{2}\sum_{x,y\in\mathcal{X}:x\neq y}\gamma
_{yx}(p).\label{EqProofSlowBound}%
\end{equation}

Set $\gamma_{\min}\doteq\min\{\Gamma_{yx}(\pi^{\ast}):\Gamma_{yx}(\pi^{\ast
})>0\}>0$ and note that $\max_{x\in\mathcal{X}}\frac{1}{\pi_{x}^{\ast}}%
<\infty$ since $\pi_{\min}^{\ast}\doteq\min_{x\in\mathcal{X}}\pi_{x}^{\ast}%
>0$. Let $x,y\in\mathcal{X}$, $x\neq y$, and set $z\doteq\frac{p_{x}\pi
_{y}^{\ast}}{p_{y}\pi_{x}^{\ast}}$. We distinguish two cases.

\emph{Case 1: }$z\geq\frac{1}{2}$\emph{ or }$\Gamma_{yx}^{\lambda}%
(p)-\Gamma_{yx}(\pi^{\ast})\geq0$\emph{.} Suppose first that $p_{y}\neq0$ and
$z\geq1/2$. Since $|\log s|\leq2|s-1|$ for all $s\geq\frac{1}{2}$,
\begin{align*}
\rho_{yx}^{\lambda}(p)  &  =p_{y}\log z\bigl(\Gamma_{yx}^{\lambda}%
(p)-\Gamma_{yx}(\pi^{\ast})\bigr)\\
&  \leq2p_{y}\left\vert \frac{p_{x}\pi_{y}^{\ast}}{p_{y}\pi_{x}^{\ast}%
}-1\right\vert \bigl|\Gamma_{yx}^{\lambda}(p)-\Gamma_{yx}(\pi^{\ast})\bigr|\\
&  =\frac{2}{\pi_{x}^{\ast}}|p_{x}\pi_{y}^{\ast}-p_{y}\pi_{x}^{\ast
}|\bigl|\Gamma_{yx}^{\lambda}(p)-\Gamma_{yx}(\pi^{\ast})\bigr|\\
&  \leq\frac{2}{\pi_{\min}^{\ast}}\left(  \pi_{y}^{\ast}|p_{x}-\pi_{x}^{\ast
}|+\pi_{x}^{\ast}|\pi_{y}^{\ast}-p_{y}|\right)  C\lambda\Vert p-\pi^{\ast
}\Vert.
\end{align*}
This inequality is trivially true if $p_{y}=0$ or if $z<1/2$ and $\Gamma
_{yx}^{\lambda}(p)-\Gamma_{yx}(\pi^{\ast})\geq0$, and thus is always valid for
Case 1.

\emph{Case 2: }$\Gamma_{yx}^{\lambda}(p)-\Gamma_{yx}(\pi^{\ast})<0$\emph{ and
}$z\in\lbrack0,\frac{1}{2})$\emph{.} Since $\log s -s+1\leq0$ for all $s\geq
0$,
\begin{align*}
\tfrac{1}{2}\gamma_{yx}(p)+\rho_{yx}^{\lambda}(p)  &  =p_{y}\left(  \tfrac
{1}{2}\left(  \log z-z+1\right)  \Gamma_{yx}(\pi^{\ast})+\log z\bigl(\Gamma
_{yx}^{\lambda}(p)-\Gamma_{yx}(\pi^{\ast})\bigr)\right) \\
&  \leq p_{y}\left(  \tfrac{1}{2}\left(  \log z-z+1\right)  \gamma_{\min
}+|\log z|2C\lambda\right) \\
&  \leq\tfrac{1}{2}p_{y}\left(  -|\log z|\left(  \gamma_{\min}-4C\lambda
\right)  +(1-z)\gamma_{\min}\right)  .
\end{align*}
This quantity is non-positive for $z\in\lbrack0,\frac{1}{2})$ whenever
$\lambda\leq\lambda_{1}\doteq\frac{\gamma_{\min}}{16C}\wedge1$. Recalling
inequality \eqref{EqProofSlowBound}, we have for $\lambda\in\lbrack
0,\lambda_{1}]$ that
\begin{align*}
&  \sum_{x,y:x\neq y}\left(  \gamma_{yx}(p)+\rho_{yx}^{\lambda}(p)\right) \\
&  \quad\leq-\frac{c}{2}\Vert p-\pi^{\ast}\Vert^{2}+\frac{2C\lambda}{\pi
_{\min}^{\ast}}\Vert p-\pi^{\ast}\Vert\sum_{x,y\in\mathcal{X}:x\neq y}\left(
\pi_{y}^{\ast}|p_{x}-\pi_{x}^{\ast}|+\pi_{x}^{\ast}|\pi_{y}^{\ast}%
-p_{y}|\right) \\
&  \quad\leq-\frac{c}{2}\Vert p-\pi^{\ast}\Vert^{2}+\frac{2C\lambda}{\pi
_{\min}^{\ast}}\Vert p-\pi^{\ast}\Vert\left(  \sum_{x\in\mathcal{X}}|p_{x}%
-\pi_{x}^{\ast}|+\sum_{y\in\mathcal{X}}|\pi_{y}^{\ast}-p_{y}|\right) \\
&  \quad\leq-\frac{c}{2}\Vert p-\pi^{\ast}\Vert^{2}+\frac{4C\lambda}{\pi
_{\min}^{\ast}}\Vert p-\pi^{\ast}\Vert^{2}.
\end{align*}
This last quantity is strictly negative if $\lambda<\lambda_{2}\doteq
\min\{\lambda_{1},c\frac{\pi_{\min}^{\ast}}{8C}\}$ and $p\neq\pi^{\ast}$, and
zero for $p=\pi^{\ast}$. Choosing $\lambda_{0}\in(0,\lambda_{2})$, we find
that \eqref{EqProofSlow} holds, with equality if and only if $p=\pi^{\ast}$.
\rule{0.5em}{0.5em}

The bound on $\lambda$ obtained in the proof is obviously conservative, and
better bounds that depend on $\Gamma(\pi^{\ast})$ and $\pi^{\ast}$ can be found.

\section{Systems of Gibbs Type}

\label{SectGibbs}

In this section we revisit the class of Gibbs models introduced in Section 4
of \cite{BuDuFiRa1}. We begin by recalling the basic definitions. Let
$K:\mathcal{X}\times\mathbb{R}^{d}\rightarrow\mathbb{R}$ be such that for each
$x\in\mathcal{X}$, $K(x,\cdot)$ is a continuously differentiable function on
$\mathbb{R}^{d}$. We sometimes write $K(x,p)$ as $K^{x}(p)$. One special case
we discuss in detail is given by%
\begin{equation}
K(x,p)=V(x)+\beta\sum_{y\in\mathcal{X}}W(x,y)p_{y},\;(x,p)\in\mathcal{X}%
\times\mathbb{R}^{d} \label{eq:affinegibbs}%
\end{equation}
where $V:\mathcal{X}\rightarrow\mathbb{R}$, $W:\mathcal{X}\times
\mathcal{X}\rightarrow\mathbb{R}$ and $\beta>0$.

Let $(\alpha(x,y))_{x,y\in\mathcal{X}}$ be an irreducible and symmetric matrix
with diagonal entries equal to zero and off-diagonal entries either one or
zero. Define $H:\mathcal{X}\times\mathbb{R}^{d}\rightarrow\mathbb{R}$ by
\begin{align}
H(x,p)\doteq H^{x}(p)&=K^{x}(p)+\sum_{z\in\mathcal{X}}\left(
\frac{\partial}{\partial p_{x}}K^{z}(p)\right)  p_{z} \nonumber\\
&= \frac{\partial}{\partial p_{x}}\left(\sum_{z\in\mathcal{X}}K^{z}(p)p_{z}  \right)
\label{def-H}
\end{align}
and $\Psi:\mathcal{X}\times\mathcal{X}\times\mathbb{R}^{d}\rightarrow
\mathbb{R}$ by%
\[
\Psi(x,y,p)\doteq H^{y}(p)-H^{x}(p),\;(x,y,p)\in\mathcal{X}\times
\mathcal{X}\times\mathbb{R}^{d}.
\]
Let
\begin{equation}
\Gamma_{x,y}(p)\doteq e^{-\left(  \Psi(x,y,p)\right)  ^{+}}\alpha(x,y),\;
x\neq y,\;p\in\mathcal{P}(\mathcal{X}), \label{eq:eqlimgen}%
\end{equation}
where recall that we identify $\mathcal{P}(\mathcal{X})$ with the simplex
$\mathcal{S}$.
%\doteq\{p\in\mathbb{R}^{d}:p_{x}\geq0\text{ and }%
%\sum_{x=1}^{d}p_{x}=1\}.$
Then for $p\in\mathcal{S}$, $\Gamma(p)$ is the generator of an ergodic
finite-state Markov process, and the unique invariant distribution on
$\mathcal{X}$ is given by $\pi(p)$ with
\begin{equation}
\pi(p)_{x}\doteq\frac{1}{Z(p)}\exp\left(  -H^{x}(p)\right)  ,
\label{ExLimitStationary}%
\end{equation}
where%
\[
Z(p)\doteq\sum_{x\in\mathcal{X}}\exp\left(  -H^{x}(p)\right)  .
\]

By studying the asymptotics of certain scaled relative entropies, the
following candidate Lyapunov function was identified in Theorem 4.2 of
\cite{BuDuFiRa1}:
\begin{equation}
F(p)=\sum_{x\in\mathcal{X}}(K^{x}(p)+\log p_{x})p_{x}\label{eqn:miss1}%
\end{equation}
for $p\in{\mathcal{S}}$. We note that in \cite{BuDuFiRa1} $K(x,\cdot)$ was
taken to be twice continuously differentiable (this property was used in the
proof of Lemma 4.1 of \cite{BuDuFiRa1}), however here we merely assume that
$K(x,\cdot)$ is $\mathcal{C}^{1}$. Also note that in the special case of
\eqref{eq:affinegibbs},
\begin{equation}
F(p)=\sum_{x\in\mathcal{X}}\left(  V(x)+\beta\sum_{y\in\mathcal{X}}%
W(x,y)p_{y}+\log p_{x}\right)  p_{x}.\label{eqn:miss2}%
\end{equation}
Since $\sum_{x\in\mathcal{X}}p_{x}\log p_{x}$ is the negative of the entropy
of $p$, $F$ is the sum of a convex function, an affine function and a
quadratic function on $\mathcal{P}(\mathcal{X})$. This fact is useful in
determining whether or not the fixed points of \eqref{EqLimitKolmogorov} are stable.

Recall the set $\mathcal{H}_{0}$ introduced in Section \ref{subs-def}
%Let $\partial_{x}$ denote $\partial/\partial
%_{p_{x}}$. From \eqref{EqLimitLyapunov} it follows that for $p\in
%\mathcal{P}(\mathcal{X})$ such that $p_{x}>0$ for all $x\in\mathcal{X}$,%
%\begin{equation}
%\partial_{x}F(p)=\log p_{x}+1+V(x)+2\beta\sum_{y\neq x}W(x,y)p_{y}%
%.\label{EqPartialF}%
%\end{equation}
%Set $\mathcal{P}_{tan}(\mathcal{X})\doteq\{\nu\in\mathbb{R}^{d}:\sum_{i=1}%
%^{d}\nu_{i}=0\}$.
and note that for $p\in\mathcal{S}^{\circ}=\{p\in\mathcal{S}:p_{x}%
>0\mbox{ for all }x=1,\ldots,d\}$ the directional derivative of the function $F$  in \eqref{eqn:miss1} in any
direction $v\in\mathcal{H}_{0}$ is given by
\begin{equation}
\frac{\partial}{\partial v}F(p)\doteq\langle DF(p),v\rangle=\sum
_{x\in\mathcal{X}}v_{x}\left(  \log p_{x}+K^{x}(p)+\sum_{y\in\mathcal{X}%
}\frac{\partial}{\partial p_{x}}K^{y}(p)\right), \label{eq:dirdef}%
\end{equation}
where we have used that $\sum_{x \in \mathcal{X}} v_x = 0$.
The following result shows that the fixed points of \eqref{EqLimitKolmogorov}
can be characterized as critical points of $F$.
%If the derivatives of $F$ in all directions in ${\mathcal{S}}^{\mbox{\tiny{tan}}}$ vanish at some point $p\in\mathcal{P}(\mathcal{X})$ then $p$ is
%an equilibrium point for Eq.~\eqref{EqLimitKolmogorov}, the nonlinear forward
%equation. The converse is true as well.

\begin{theorem}
\label{ThLyapunovGradient} Let $\Gamma$ be as defined in \eqref{eq:eqlimgen}
and $p\in\mathcal{P}(\mathcal{X})$. Then $p$ is a fixed point for
\eqref{EqLimitKolmogorov} if and only if $p\in\mathcal{S}^{\circ}$ and
$\frac{\partial}{\partial v}F(p)=0$ for all $v\in\mathcal{H}_{0}$.
\end{theorem}

\begin{proof}
Recall that $\pi(p)$ is the unique invariant probability associated with
$\Gamma(p)$, and hence $\pi(p)\Gamma(p)=0$. Also note that $p$ is a fixed
point for \eqref{EqLimitKolmogorov} if and only if $p\Gamma(p)=0$, which,
since $\Gamma(p)$ is a rate matrix of an ergodic Markov process, can be true
if and only if $p=\pi(p)$. Since $\pi(p)\in\mathcal{S}^{\circ}$ for every
$p\in\mathcal{S}$ we have that any fixed point of \eqref{EqLimitKolmogorov} is
in $\mathcal{S}^{\circ}$. For $x,y\in\mathcal{X}$, $x\neq y$, let $v
^{x,y}\doteq e_{x}-e_{y}$, where $e_{x}$ is the unit vector in direction $x$.
Then by (\ref{eq:dirdef}), (\ref{def-H}) and (\ref{ExLimitStationary}), for
any $p\in\mathcal{S}^{\circ}$%
\begin{align}
\frac{\partial}{\partial v^{x,y}}F(p) &  =\log p_{x}-\log p_{y}+K^{x}(p)-K^{y}(p)\nonumber\\
&  \quad+\sum_{z\in\mathcal{X}}\left(  \frac{\partial}{\partial p_{x}}%
K^{z}(p)-\frac{\partial}{\partial p_{y}}K^{z}(p)\right)  p_{z}\nonumber\\
&  =\log p_{x}-\log p_{y} + \left(  H^{x}(p)-H^{y}(p)\right)  \nonumber\\
&  =\log\left(  \frac{p_{x}}{p_{y}}\right)  -\log\left(  \frac{\pi(p)_{x}}%
{\pi(p)_{y}}\right)  .\label{EqProofGradient}%
\end{align}
%where the last equality is from \eqref{ExLimitStationary}.
If $p$ is a fixed point of \eqref{EqLimitKolmogorov} then $p=\pi(p)$, and so
$\frac{\partial}{\partial v^{x,y}}F(p)=0$ for all $x,y\in\mathcal{X}$, $x\neq
y$. From this it follows that $\frac{\partial}{\partial v}F(p)=0$ for all
$v\in\mathcal{H}_{0}$.

Conversely, suppose $p\in\mathcal{S}^{\circ}$ and $\frac{\partial}{\partial
v}F(p)=0$ for all $v\in\mathcal{H}_{0}$. Then from
\eqref{EqProofGradient}
\[
\frac{p_{x}}{p_{y}}=\frac{\pi(p)_{x}}{\pi(p)_{y}}\mbox{ for all }x,y\in
\mathcal{X}.
\]
Thus $p=\pi(p)$, which says that $p$ is a fixed point of \eqref{EqLimitKolmogorov}.
\end{proof}

\medskip According to Theorem~\ref{ThLyapunovGradient}, the equilibrium points
of the forward equation \eqref{EqLimitKolmogorov} are precisely the critical
points of $F$ on $\mathcal{P}(\mathcal{X})$. Note that although for each $p$,
$\Gamma(p)$ is a rate matrix of a Markov process with a unique invariant
measure, the dynamical system \eqref{EqLimitKolmogorov} can have multiple
stable and unstable equilibria. Here is an example.

\begin{example}
\label{ExmplMultEquilibria}\emph{\ Assume that }$\mathcal{X}=\{1,2\}$\emph{,
and} $K$ \emph{is given as in \eqref{eq:affinegibbs}} with $V\equiv
0$, $W(1,1)=0=W(2,2)$, \emph{and} $W(1,2)=1=W(2,1)$\emph{. Then }$F(p)=f(p_{1})$\emph{\ with }%
\[
f(x)\doteq x\log x+(1-x)\log(1-x)+2\beta(1-x)x,\quad x\in\lbrack0,1].
\]
\emph{The critical points of }$F$\emph{\ on }$\mathcal{P}(\{1,2\})$\emph{\ are
in a one-to-one correspondence with the critical points of }$f$\emph{\ on
}$[0,1]$\emph{. We have }$f(0)=0=f(1)$\emph{, and for }$x\in(0,1)$\emph{ }%
\[
f^{\prime}(x)=\log x -\log(1-x)+2\beta-4\beta x,f^{\prime\prime}(x)=\frac
{1}{x}+\frac{1}{1-x}-4\beta.
\]
\emph{Moreover, }$f^{\prime}(x)\rightarrow-\infty$\emph{\ as }$x$\emph{\ tends
to zero, and }$f^{\prime}(x)\rightarrow\infty$\emph{\ as }$x$\emph{\ tends to
one.}

\emph{If }$\beta\leq1$\emph{\ then }$f$\emph{\ has exactly one critical point,
namely a global minimum at }$x=\frac{1}{2}$\emph{. If }$\beta>1$\emph{\ then
there are three critical points, one local maximum at }$x=\frac{1}{2}%
$\emph{\ and two minima at }$x_{\beta}$\emph{\ and }$1-x_{\beta}$\emph{,
respectively, for some }$x_{\beta}\in(0,\frac{1}{2})$\emph{, where }$x_{\beta
}\rightarrow\frac{1}{2}$\emph{\ as }$\beta\downarrow1$\emph{, }$x_{\beta
}\rightarrow0$\emph{\ as }$\beta$\emph{\ goes to infinity. The two minima of
}$f$\emph{\ correspond to stable equilibria of the forward equation, while the
local maximum corresponds to an unstable equilibrium.}
\end{example}

\subsection{Lyapunov function property}

\label{SectGibbsLyapunov}Suppose that the function $F$ defined in
(\ref{eqn:miss1}) is positive definite (in the sense of Definition
\ref{posdef}) in a neighborhood of a fixed point $\pi^{\ast}$ of
\eqref{EqLimitKolmogorov} which contains no other fixed point of
\eqref{EqLimitKolmogorov}. In this section we show that $F$ is a local
Lyapunov function for the ODE \eqref{EqLimitKolmogorov} (associated with the
neighborhood and the fixed point $\pi^{\ast}$), with $\Gamma$ defined by
\eqref{eq:eqlimgen}. This result is an immediate consequence of the theorem
below and Definition \ref{loclyap}. Together with Proposition \ref{stab-loc}
this will imply $\pi^{\ast}$ is locally stable.

\begin{theorem}
\label{ThLyapunovProp} Let $p(\cdot)$ be a solution to the forward equation
\eqref{EqLimitKolmogorov} with $\Gamma$ as defined in (\ref{eq:eqlimgen}) and
some initial distribution $p(0)\in\mathcal{P}(\mathcal{X})$. Then for all
$t\geq0$,
\begin{equation}
\frac{d}{dt}F(p(t))=\left.  \frac{d}{dt}R\left(  p(t)\Vert\pi(q)\right)
\right\vert _{q=p(t)}\leq0. \label{EqDescent}%
\end{equation}
Moreover, $\frac{d}{dt}F(p(t))=0$ if and only if $p(t)=\pi(p(t))$.
\end{theorem}

\begin{proof}
We will show that if $p(\cdot)$ is the solution to \eqref{EqLimitKolmogorov}
with $p(0)=q$ then
\begin{equation}
\left.  \frac{d}{dt}F(p(t))\right\vert _{t=0}=\left.  \frac{d}{dt}R\left(
p(t)\Vert\pi(q)\right)  \right\vert _{t=0}. \label{EqProofLyapunov}%
\end{equation}
In view of the semigroup property of solutions to the \mbox{ODE}
\eqref{EqLimitKolmogorov}, and since $q$ is arbitrary, the validity of
\eqref{EqProofLyapunov} implies the first equality in \eqref{EqDescent}.
%$\frac{d}{dt}F(p(t))=\frac{d}{dt}R\left(
%p(t)\Vert\pi(q)\right)  $, $q=p(t)$, for all $t\geq0$.

Let $p(0)=q$. By the definition (\ref{eqn:miss1}) of $F$ and since $\sum
_{x\in\mathcal{X}}\frac{dp_{x}}{dt}(0)=0$,
\[
\left.  \frac{d}{dt}F(p(t))\right\vert _{t=0}=\sum_{x\in\mathcal{X}}\log q_{x}\frac{dp_{x}}{dt}(0)+\sum_{x\in\mathcal{X}}K^{x}(q)\frac{dp_{x}}%
{dt}(0)+\sum_{x,y\in\mathcal{X}}q_{x}\frac{\partial}{\partial p_{y}}%
K^{x}(q)\frac{dp_{y}}{dt}(0).
\]
On the other hand, by the definition of relative entropy,
(\ref{ExLimitStationary}) and (\ref{def-H}), and again using the relation
$\sum_{x\in\mathcal{X}}\frac{dp_{x}}{dt}(0)=0$, we have
\begin{align}
\left.  \frac{d}{dt}R\left(  p(t)\Vert\pi(q)\right)  \right\vert _{t=0} &
=\left.  \frac{d}{dt}\left(  \sum_{x\in{\mathcal{X}}}p_{x}(t)\log
p_{x}(t)\right)  \right\vert _{t=0}-\sum_{x\in{\mathcal{X}}}\frac{dp_{x}}%
{dt}(0)\log  \pi_{x}(q) \label{line1}\\
&  =\sum_{x\in\mathcal{X}}\log q_{x}\frac{dp_{x}}{dt}(0)+\sum_{x\in
\mathcal{X}}K^{x}(q)\frac{dp_{x}}{dt}(0)\nonumber\\
&  \quad+\sum_{x,z\in\mathcal{X}}q_{z}\frac{\partial}{\partial p_{x}}%
K^{z}(q)\frac{dp_{x}}{dt}(0).\nonumber
\end{align}
Comparing the right sides of the last two displays we see that
\eqref{EqProofLyapunov} holds.

The rest of the assertion follows from the observation that $\pi(q)$ is the
stationary distribution for the (linear) Markov family associated with
$\Gamma(q)$ and from the Lyapunov property of relative entropy in the case of
ergodic (linear) Markov processes; see Lemma 3.1 in \cite{BuDuFiRa1}.
\end{proof}

% \begin{remark}\footnote{Is there a compelling reason to include this remark? We don't develop this further in any way and as such it seems to distract from the theme of this section.}
% \emph{Consider the ODE \eqref{EqLimitKolmogorov} with $\Gamma$ taking a
% general form (i.e., not necessarily given by \eqref{eq:eqlimgen}). 
% Suppose for all $p \in \mathcal{P}(\mathcal{X})$, $\Gamma(p)$ is ergodic with the unique invariant measure $\pi(p)$.
% Suppose
% there is a continuously differentiable function $G$ such that
% \[
% \sum_{x\in\mathcal{X}}\frac{\partial}{\partial p_{x}}G(p)\cdot\nu_{x}%
% =-\sum_{x\in\mathcal{X}}\log\left(  \pi_{x}(p)\right)  \cdot\nu_{x}%
% \]
% for all $p\in\mathcal{S}^{\circ}$, all $\nu\in\mathcal{H}_{0}.$ In view of
% \eqref{line1} and the fact that $\frac{d}{dt}p_{x}(t)$ lies in }%
% $\mathcal{H}_{0}$\emph{, the function
% \[
% F(p)\doteq G(p)+\sum_{x\in\mathcal{X}}p_{x}\log(p_{x})
% \]
% satisfies $\frac{d}{dt}F(p(t))\leq0$ for  all
% $t\geq0$. Thus $G$ defines a candidate Lyapunov function for
% \eqref{EqLimitKolmogorov}.}
% \end{remark}

\begin{remark}
\emph{Consider the slow adaptation setting of Section~\ref{SectSlowAdaptation}
for the Gibbs model with $K$ as in \eqref{eq:affinegibbs}. Thus we start from
a family of rate matrices }$\Gamma(p)$\emph{,} $p\in\mathcal{P}(\mathcal{X}%
)$\emph{, defined according to \eqref{eq:eqlimgen}. Suppose that }$\pi^{\ast}%
$\emph{\ is a fixed point of the mapping }$p\mapsto\pi(p).$\emph{\ For
}$\lambda\in\lbrack0,1]$\emph{,} $p\in\mathcal{P}(\mathcal{X})$\emph{, set
}$\Gamma^{\lambda}(p)\doteq\Gamma(\pi^{\ast}+\lambda(p-\pi^{\ast}))$\emph{.
The rate matrices }$\Gamma^{\lambda}(p)$\emph{\ are again of Gibbs type, that
is, }$\Gamma^{\lambda}(p)$\emph{\ satisfies \eqref{eq:eqlimgen}, but with
$\Psi$ replaced by $\Psi^{\lambda}$, where $\Psi^{\lambda}$ is defined exactly
as $\Psi$ is with $K$ as in \eqref{eq:affinegibbs}, but with different
potentials in place of }$V$\emph{ and }$W$\emph{. In particular, the
potentials }$V^{\lambda,\beta}$\emph{, }$W^{\lambda}$\emph{\ are given by }%
\[
V^{\lambda,\beta}(x)\doteq V(x)+2\beta(1-\lambda)\sum_{z\in\mathcal{X}%
}W(x,z)\pi_{z}^{\ast},\quad W^{\lambda}(x,y)\doteq\lambda W(x,y).
\]
\emph{Fix }$\lambda\geq0$\emph{. Then (\ref{eqn:miss2}) and
Theorem~\ref{ThLyapunovProp} imply that if the function }%
\begin{align*}
F^{\lambda}(p) &  \doteq\sum_{x\in\mathcal{X}}p_{x}\log p_{x}+\sum
_{x\in\mathcal{X}}V^{\lambda,\beta}(x)p_{x}+\beta\sum_{x,y\in\mathcal{X}%
}W^{\lambda}(x,y)p_{x}p_{y}\\
&  =\sum_{x\in\mathcal{X}}p_{x}\log p_{x}+\sum_{x\in\mathcal{X}}\left(
V(x)+2\beta(1-\lambda)\sum_{z\in\mathcal{X}}W(x,z)\pi_{z}^{\ast}\right)
p_{x}\\
&  \quad\mbox{}+\lambda\beta\sum_{x,y\in\mathcal{X}}W(x,y)p_{x}p_{y},
\end{align*}
\emph{is positive definite in some neighborhood of $\pi^{\ast}$, then it is a local Lyapunov function for
\eqref{EqSlowKolmogorov} (associated with that neighborhood and the fixed point $\pi^{\ast}$). } \emph{Proposition \ref{PropSlow}, on the other
hand, implies that }$\bar{F}(p)\doteq R\left(  p\Vert\pi^{\ast}\right)  $
\emph{is also a local Lyapunov function when} $\lambda$\emph{ is positive but
sufficiently small. By the definition of relative entropy,
\eqref{ExLimitStationary}, \eqref{def-H} and \eqref{eq:affinegibbs},}
\begin{align*}
\bar{F}(p) &  =R\left(  p\Vert\pi^{\ast}\right)  \\
&  =\sum_{x\in\mathcal{X}}p_{x}\log p_{x}-\sum_{x\in\mathcal{X}}p_{x}\log
\pi_{x}^{\ast}\\
&  =\sum_{x\in\mathcal{X}}p_{x}\log p_{x}+\log Z(\pi^{\ast})+\sum
_{x\in\mathcal{X}}\left(  V(x)+2\beta\sum_{z\in\mathcal{X}}W(x,z)\pi_{z}%
^{\ast}\right)  p_{x},
\end{align*}
\emph{which is equal to }$F^{\lambda}(p)+\log Z(\pi^{\ast})$\emph{\ for
}$\lambda=0$\emph{. Observe that the term }$\log Z(\pi^{\ast})$\emph{\ has no
impact on the Lyapunov function property as it does not depend on }$p$\emph{.}
\emph{Thus, the function $F^{\lambda}$ includes \textquotedblleft correction
terms\textquotedblright\ (that vanish when $\lambda=0$) and serves as a
Lyapunov function(when positive definite) not just for small $\lambda$ but rather for all $\lambda
\in(0,1]$.}
\end{remark}

%\textcolor{red}{

%}

%\footnote{If we can have a simple analytical example, say even for
%${\mathcal X} = \{1, 2, 3\}$ that would be nice.}

\subsection{Comparison with existing results for It{\^{o}} diffusions}

\label{SectGibbsComparison}

A situation analogous to that of this section is considered in \cite{tamura87}%
, where the author studies the long-time behavior of \textquotedblleft
nonlinear\textquotedblright\ It{\^{o}}-McKean diffusions of the form%
\begin{equation}
dX(t)=-\left(  \nabla V\bigl(X(t)\bigr)+2\beta\int_{\mathbb{R}^{d}}\nabla
_{1}W\bigl(X(t),y)\bigr)\mu_{t}(dy)\right)  dt+\sqrt{2}dB(t),
\label{EqLimitSDE}%
\end{equation}
where $\mu_{t}$ is the probability law of $X(t)$, $B$ is a standard
$d$-dimensional Wiener process, $V$ a function $\mathbb{R}^{d}\mapsto
\mathbb{R}$, the environment potential, and $W$ a \emph{symmetric} function
$\mathbb{R}^{d}\times\mathbb{R}^{d}\mapsto\mathbb{R}$ with zero diagonal, the
interaction potential. Here $\nabla_{1}$ denotes gradient with respect to the
first $\mathbb{R}^{d}$-valued variable. Signs and constants have been chosen
in analogy with the finite-state models considered here. Solutions of
\eqref{EqLimitSDE} arise as weak limits of the empirical measure processes
associated with weakly interacting It{\^{o}} diffusions. The $N$-particle
model is described by the system
\[
dX^{i,N}(t)=-\nabla V\bigl(X^{i,N}(t)\bigr)dt-\frac{2\beta}{N}\sum_{j=1}%
^{N}\nabla_{1}W\bigl(X^{i,N}(t),X^{j,N}(t)\bigr)dt+\sqrt{2}dB^{i}(t),
\]
where $i\in\{1,\ldots,N\}$, $B^{1},\ldots,B^{N}$ are independent standard
Brownian motions.

In \cite{tamura87} a candidate Lyapunov function $F\!:\mathcal{P}%
(\mathbb{R}^{d})\rightarrow\lbrack0,\infty]$, referred to as the
\textquotedblleft free energy function\textquotedblright, is introduced
without explicit motivation, and then shown to be in fact a valid Lyapunov
function. The same function is also considered in \cite{CaMcVi} and plays a
key role in their study of convergence properties of $\mu_{t}$ as
$t\rightarrow\infty$. The function takes the following form. If $\mu$ is a
probability measure that is absolutely continuous with respect to Lebesgue
measure and of the form $f_{\mu}(x)dx$, then
\begin{equation}
F(\mu)\doteq\int\log   f_{\mu}(x)  f_{\mu}(x)dx+\int
V(x)\mu(dx)+\int\int W(x,y)\mu(dx)\mu(dy).\label{ExTamuraLyapunov}%
\end{equation}
In all other cases $F(\mu)\doteq\infty$. This function is clearly a close
analogue of the function in (\ref{eqn:miss1}), which was derived as the limit
of scaled relative entropies. There are, however, some interesting differences
in the presentation and proof of the needed properties. The most significant
of these is how one represents the derivative of the composition of the
Lyapunov function with the solution to the forward equation. In
\cite{tamura87} the descent property is established by expressing the orbital
derivatives of $F$ in terms of the Donsker-Varadhan rate function associated
with the empirical measures of solutions to \eqref{EqLimitSDE}, when the
measure $\mu_{t}$ is frozen at $\mu\in\mathcal{P}(\mathbb{R}^{d})$. In
contrast, in our case the orbital derivative of the Lyapunov function is
expressed as the orbital derivative of relative entropy with respect to the
invariant distribution $\pi(p)$ that is obtained when the dynamics of the
nonlinear Markov process are frozen at $p$. The latter expression also applies to
the diffusion case in the sense that for all $t\geq0$,
\begin{equation}
\frac{d}{dt}F(\mu_{t})=\frac{d}{dt}R\left(  \mu_{t}\Vert\pi_{\nu}\right)
_{\nu=\mu_{t}},\label{EqTamuraREDerivative}%
\end{equation}
where $\mu_{t}$ is the law of $X(t)$, $X$ being the solution to
\eqref{EqLimitSDE} for some (absolutely continuous) initial condition, and
$\pi_{\nu}\in\mathcal{P}(\mathbb{R}^{d})$ is given by
\begin{equation}
\pi_{\nu}(dx)\doteq\frac{1}{Z_{\nu}}\exp\left(  -V(x)-2\beta\int
_{\mathbb{R}^{d}}W(x,y)\nu(dy)\right)  dx,\label{ExTamuraStationary}%
\end{equation}
with $Z_{\nu}$ the normalizing constant. Clearly, the probability measures
given by \eqref{ExTamuraStationary} correspond to the distributions $\pi
(p)\in\mathcal{P}(\mathcal{X})$ defined in \eqref{ExLimitStationary}. The
relationship \eqref{EqTamuraREDerivative} can be established in a way
analogous to the proof of Theorem~\ref{ThLyapunovGradient}. On the other hand,
the representation for the orbital derivative of the Lyapunov function in
terms of the Donsker-Varadhan rate function as established in \cite{tamura87}
for the diffusion case does \emph{not} carry over to the finite-state Gibbs
models studied above. Because of this, we argue that
(\ref{EqTamuraREDerivative}) is the more natural and general way to
demonstrate that $F$ has the properties required of a Lyapunov function.

To make this more precise, consider the case of linear Markov processes. Let
$\Gamma$ be the infinitesimal generator (rate matrix) of an $\mathcal{X}%
$-valued ergodic Markov family with unique stationary distribution $\pi
\in\mathcal{P}(\mathcal{X})$. Let $p(\cdot)$ be a solution of the
corresponding forward equation \eqref{EqLimitKolmogorov}. Then
\[
\frac{d}{dt}R\left(  p(t)\Vert\pi\right)  =\int_{\mathcal{X}}f_{t}%
\Gamma\left(  \log f_{t}\right)  d\pi=-\mathcal{E}_{\Gamma}\left(  f_{t},\log
f_{t}\right)  ,
\]
where $f_{t}\doteq\frac{dp(t)}{d\pi}$ is the density of $p(t)$ with respect to
$\pi$ and $\mathcal{E}_{\Gamma}(\cdot,\cdot)$ is the Dirichlet form associated
with $\Gamma$ and its stationary distribution $\pi$, that is,
\[
\mathcal{E}_{\Gamma}(f,g)\doteq-\int_{\mathcal{X}}f\,\Gamma(g)d\pi=-\sum
_{x\in\mathcal{X}}f(x)\left(  \sum_{y\in\mathcal{X}}g(y)\Gamma_{xy}\right)
\pi_{x}%
\]
for test functions $f,g\!:\mathcal{X}\rightarrow\mathbb{R}$. On the other
hand, the Donsker-Varadhan $I$-function associated with $\Gamma$ is given by
\[
I_{\Gamma}(\mu)=\sup_{f:\mathcal{X}\rightarrow(0,\infty)}\left(-\int_{\mathcal{X}%
}\frac{\Gamma(f)(x)}{f(x)}\mu(dx)\right),\quad\mu\in\mathcal{P}(\mathcal{X}).
\]
If $\Gamma$ is also reversible (i.e., $\Gamma_{xy}\pi_{x}=\Gamma_{yx}\pi_{y}$
for all $x,y\in\mathcal{X}$), then $I$ takes the more explicit form
\[
I_{\Gamma}(\mu)=\mathcal{E}_{\Gamma}\left(  \sqrt{\frac{d\mu}{d\pi}}%
,\sqrt{\frac{d\mu}{d\pi}}\right)  ;
\]
see, for instance, Theorem~IV.14 and Exercise IV.24 in \cite[pp.\,47-50]%
{hollander00}. In general, the functions $f\mapsto\mathcal{E}_{\Gamma}\left(
\sqrt{f},\sqrt{f}\right)  $ and $f\mapsto\mathcal{E}_{\Gamma}\left(
f,\log f\right)  $ with $f$ ranging over all non-degenerate $\pi$-densities
are not proportional. As a counterexample, it is enough to evaluate the
Dirichlet forms for $\Gamma=%
\begin{pmatrix}
-1 & 1\\
1 & -1
\end{pmatrix}
$ and $\pi=(1/2,1/2)$.

\section{A PDE for Limits of Relative Entropies}

\label{subs-timedeppde} In the last section we saw that the scaling limits of
relative entropies with respect to stationary distributions of certain
$N$-particle Markov processes ${\boldsymbol{X}}^{N}$ yield candidate Lyapunov
functions for \eqref{EqLimitKolmogorov}. In this section we consider the case
where closed form expressions for the stationary distributions are not
available and consequently these limits cannot be evaluated explicitly. Recall
from the discussion in Section 5 of \cite{BuDuFiRa1} that in such cases our
basic approach to constructing Lyapunov functions is to take limits of the
scaled relative entropy $F_{t}^{N}$ specified in equation (1.4) of
\cite{BuDuFiRa1} (see also equation \eqref{eq:fnq} in this section), first as
$N\rightarrow\infty$ and then as $t\rightarrow\infty$. Theorem 5.5 of
\cite{BuDuFiRa1} shows that under some basic assumptions, the limit as
$N\rightarrow\infty$ coincides with the large deviation rate function
$J_{t}(\cdot):{\mathcal{S}}\rightarrow\lbrack0,\infty)$ for a certain sequence of empirical measures of $N$-particle systems
that converge to the solution of the ODE \eqref{EqLimitKolmogorov}.
This large deviations result\cite{Leo, BorSun, DupRamWu12} is recalled in Section \ref{Subs-LD}.
%  We begin by recalling
% in Section \ref{Subs-LD} a large deviation result established in
% \cite{DupRamWu12}  
Next, we
formally derive a time-dependent PDE for the associated large deviation rate
function $J_{t}(r)$ in Section \ref{sec:timdeppde}, and present the stationary
version of this PDE in Section \ref{sec-PDE}. These formal calculations simply
motivate the form of the PDE -- the main result presented in Section
\ref{subs-pdelyap} that subsolutions to the stationary PDE serve as local
Lyapunov functions for the ODE \eqref{EqLimitKolmogorov}, does not rely on
this derivation. The proof of the main result relies on certain properties
that are first established in Section \ref{sec:stpde}.

\subsection{A large deviation result}

\label{Subs-LD}

Let, as in Section 2 of \cite{BuDuFiRa1}, ${\boldsymbol{X}}^{N}=(X^{1,N}%
,\ldots,X^{N,N})$ be a ${\mathcal{X}}^{N}$-valued Markov process with
transitions governed by the family of matrices $\{\Gamma(r),r\in{\mathcal{P}%
}({\mathcal{X}})\}$, where for each $r\in{\mathcal{P}}({\mathcal{X}})$,
$\Gamma(r)=\{\Gamma_{x,y}(r),x,y\in{\mathcal{X}}\}$ is a transition rate
matrix of a continuous time Markov chain on ${\mathcal{X}}$ (here for
simplicity we assume that $\Gamma^{N}=\Gamma$). Specifically, the transition
mechanism is as follows.
%Roughly speaking, for $y\neq x$,
%$\Gamma_{xy}^{N}(r)\geq0$ represents the rate at which a single particle
%transitions from state $x$ to state $y$ when the empirical measure has value
%$r$. More precisely, the transition mechanism of $\boldsymbol{X}^{N}$ is as
%follows.
Given $\boldsymbol{X}^{N}(t)=\boldsymbol{x}\in\mathcal{X}^{N}$, an index
$i\in\left\{  1,\ldots,N\right\}  $ and $y\neq x_{i}$, the jump rate at time
$t$ for the transition
\[
\left(  x_{1},\ldots,x_{i-1},x_{i},x_{i+1},\ldots,x_{N}\right)  \mapsto\left(
x_{1},\ldots,x_{i-1},y,x_{i+1},\ldots,x_{N}\right)
\]
is $\Gamma_{x_{i}y}(r^{N}(\boldsymbol{x}))$, where
%$r^{N}(\boldsymbol{x}%
%)\in{\mathcal{P}}({\mathcal{X}})$ represents the \textquotedblleft
%type\textquotedblright\ of the vector $\boldsymbol{x}\in{\mathcal{X}}^{N}$,
%which is given explicitly by%
\begin{equation}
r_{y}^{N}(\boldsymbol{x})\doteq\frac{1}{N}\sum_{i=1}^{N}1_{\{x_{i}=y\}},\qquad
y\in{\mathcal{X}}.\label{def-rn}%
\end{equation}
The jump rates for transitions of any other type are zero.
%Note that
%$r^{N}(\boldsymbol{X}^{N}(t))$ equals the empirical measure $\mu^{N}(t)$,
%defined in (\ref{def-mun}).
Under the assumption of exchangeability of the initial random vector
$\{X^{i,N}(0)\}_{i=1,\ldots,N}$ $\ $we have that the processes $\{X^{i,N}%
\}_{i=1,\ldots,N}$ are also exchangeable. From this, it follows that the
empirical measure process $\mu^{N}=\{\mu^{N}(t)\}_{t\geq0}$ is a Markov chain
taking values in $\mathcal{S}_{N}=\mathcal{S}\cap\frac{1}{N}\mathbb{Z}^{d}$,
where $\mathcal{S}$ is the unit simplex which is identified with
$\mathcal{P}({\mathcal{X}})$, with the generator ${\mathcal{L}}^{N}$ given by
\begin{equation}
{\mathcal{L}}^{N}f(r)=\sum_{x,y\in{\mathcal{X}}:x\neq y}Nr_{x}\Gamma
_{xy}(r)\left[  f\left(  r+\frac{1}{N}(e_{y}-e_{x})\right)  -f(r)\right]
\label{def-ln}%
\end{equation}
for real-valued functions $f$ on $\mathcal{S}_{N}$.

We recall the following locally uniform LDP for the empirical measure process. The LDP has been established in
\cite{Leo, BorSun} while the locally uniform version used here is  taken from
\cite{DupRamWu12}.
\begin{theorem}
\label{thm:ldpthm} 	Suppose that for each $p\in{\mathcal{S}}$, $\Gamma(p)$ is
	 the transition rate matrix of an ergodic Markov chain and that Condition \ref{cond:lip} holds.
For $t\in\lbrack0,\infty)$ let $\boldsymbol{p}^{N}(t)$ be the distribution of
${\boldsymbol{X}}^{N}(t)=(X^{1,N}(t),\ldots,X^{N,N}(t))$.
Recall the mapping $r^{N}:\mathcal{X}^{N}%
\rightarrow\mathcal{P}_{N}(\mathcal{X})$ given by (\ref{def-rn}), i.e.,
$r^{N}(\boldsymbol{x})$ is the empirical measure of $\boldsymbol{x}$ . Assume
 that the initial random vector $\{X^{i,N}(0)\}_{i=1,\ldots,N}$ is exchangeable
 and assume that $r^{N}$ under the distribution $\boldsymbol{p}^{N}(0)$
 satisfies a large deviation principle (LDP)
with a rate function $J_{0}$. Then for each $t\in\lbrack0,\infty)$, $r^{N}$
under the distribution $\boldsymbol{p}^{N}(t)$ satisfies a locally uniform LDP
on ${\mathcal{P}}({\mathcal{X}})$ with a rate function $J_{t}$, thus given any
sequence $\{q_{N}\}_{N\in\mathbb{N}}$, $q_{N}\in\mathcal{S}_{N}$, such that
$q_{N}\rightarrow q\in\mathcal{S}$,
\[
\lim_{N\rightarrow\infty}\frac{1}{N}\log\boldsymbol{p}^{N}(t)\left(  \left\{
\boldsymbol{y}\in{\mathcal{X}}^{N}:r^{N}(\boldsymbol{y})=q_{N}\right\}
\right)  =-J_{t}(q).
\]
Furthermore, $J_{t}(q)<\infty$ for all $q\in{\mathcal{P}}({\mathcal{X}})$.
\end{theorem}

We will now formally derive a PDE solved by $J_{t}(q)$. 
% It should be
% emphasized that the formal derivation of the time-dependent PDE serves simply
% to motivate the form of the stationary PDE \eqref{eq:statPDE}, and that the
% proof of the main results do not rely on this derivation.

\subsection{A time-dependent PDE}

\label{sec:timdeppde}

%Recall that we identify ${\mathcal P}({\mathcal X})$ wth the unit
%simplex, and let ${\mathcal S}_N = \frac{1}{N} \mathbb{Z}^d \cap
%{\mathcal S}$.
For notational convenience, throughout this section for $t\geq0$ and
$r\in{\mathcal{S}}$ we write $J_{t}(r)$ as $J(r,t)$. For $t\geq0$, let
$u^{N}(t)$ denote the distribution of $\mu^{N}(t)$, that is, for
$r\in{\mathcal{S}}_{N}$, let $u_{r}^{N}(t)=P\left(  \mu^{N}(t)=r\right)  $.
Then, $u^{N}$ satisfies the Kolmogorov forward equation
\begin{equation}
\dfrac{du^{N}}{dt}(t)=u^{N}(t){\mathcal{L}}^{N}, \label{un-fwd}%
\end{equation}
where ${\mathcal{L}}^{N}$ is as in \eqref{def-ln}. For $r\in{\mathcal{S}}_{N}%
$, substituting into (\ref{un-fwd}) the approximation
\[
u_{r}^{N}(t)\approx e^{-NJ(r,t)}%
\]
that follows from the LDP stated in Theorem \ref{thm:ldpthm}, and recalling
the form of the generator ${\mathcal{L}}^{N}$ from (\ref{def-ln}), we obtain
\begin{align*}
\frac{\partial}{\partial t}e^{-NJ(r,t)}  &  =\sum_{\overset{x,y\in
{\mathcal{X}},x\neq y:}{r-\frac{1}{N}(e_{y}-e_{x})\in{\mathcal{S}}_{N}}%
}e^{-NJ(r-\frac{1}{N}(e_{y}-e_{x}),t)}\left(  Nr_{x}+1\right)  \Gamma
_{xy}\left(  r-\frac{1}{N}(e_{y}-e_{x})\right) \\
&  \quad\quad-\sum_{\overset{x,y\in{\mathcal{X}},x\neq y:}{r+\frac{1}{N}%
(e_{y}-e_{x})\in{\mathcal{S}}_{N}}}e^{-NJ(r,t)}Nr_{x}\Gamma_{xy}(r).
\end{align*}
Observing that the left-hand side of the last display equals $-Ne^{-NJ(r,t)}%
\left(  \partial J(r,t)/\partial t\right)  $, and multiplying both sides by
$e^{NJ(r,t)}/N$, we obtain
\[%
\begin{array}
[c]{l}%
\displaystyle-\dfrac{\partial}{\partial t}J(r,t)\\
\displaystyle\qquad=\sum_{\overset{x,y\in{\mathcal{X}},x\neq y:}{r-\frac{1}%
{N}(e_{y}-e_{x})\in{\mathcal{S}}_{N}}}e^{-N\left[  J(r-\frac{1}{N}(e_{y}%
-e_{x}),t)-J(r,t)\right]  }\left(  r_{x}+\frac{1}{N}\right)  \Gamma
_{xy}\left(  r-\frac{1}{N}(e_{y}-e_{x})\right) \\
\displaystyle\qquad\quad\quad-\sum_{\overset{x,y\in{\mathcal{X}},x\neq
y:}{r+\frac{1}{N}(e_{y}-e_{x})\in{\mathcal{S}}_{N}}}r_{x}\Gamma_{xy}(r).
\end{array}
\]
If $J$ is smooth and $N$ is large, we can use the approximation
\[
J\left(  r-\frac{1}{N}(e_{y}-e_{x}),t\right)  -J(r,t)\approx-\frac{1}%
{N}\langle DJ(r,t),e_{y}-e_{x}\rangle+o(1/N).
\]
Substituting this approximation into the previous display, sending
$N\rightarrow\infty$ and recalling that $\Gamma_{x,y}(\cdot)$ is continuous,
we obtain the following PDE for $J(r,t)$: for $r\in{\mathcal{S}}$ and
$t\in\lbrack0,\infty)$,
\begin{equation}
-\frac{\partial J}{\partial t}(r,t)=\sum_{x,y\in{\mathcal{X}}:x\neq y}\left[
e^{\langle DJ(r,t),e_{y}-e_{x}\rangle}-1\right]  r_{x}\Gamma_{xy}(r).
\label{pde-td}%
\end{equation}
As mentioned earlier, this derivation is not rigorous because we did not
establish the smoothness properties of $J$ assumed in the calculations. In
fact, in general one does not expect this smoothness property to hold and one
would have to interpret $J$ as a viscosity solution to the PDE (\ref{pde-td})
with appropriate boundary conditions. However, the derivation simply serves to
motivate the form of the stationary PDE and the proof of the main result given
in the next section does not rely on this derivation.

\subsection{The stationary PDE and Lyapunov functions}

\label{sec-PDE}

We now introduce our main tool for constructing local Lyapunov functions for
\eqref{EqLimitKolmogorov}. Recall
\begin{equation}
F_{t}^{N}(q)\doteq\frac{1}{N}R(\otimes^{N}q\Vert\boldsymbol{p}^{N}%
(t)),\;q\in\mathcal{S}. \label{eq:fnq}%
\end{equation}
Formally writing
\begin{equation}
\lim_{t\rightarrow\infty}\lim_{N\rightarrow\infty}F_{t}^{N}(q)=\lim
_{t\rightarrow\infty}J_{t}(q)=J^{\ast}(q),\;q\in\mathcal{S}, \label{eq:852}%
\end{equation}
one expects from the formal derivation of the last section that the function
$J^{\ast}$ solves the following PDE:
\begin{equation}
\boldsymbol{H}(r,-DJ(r))=0, \label{eq:statpde}%
\end{equation}
where for $(r,\alpha)\in\mathcal{S}\times\mathbb{R}^{d}$,
\begin{equation}
\boldsymbol{H}(r,\alpha)\doteq-\sum_{x,y\in{\mathcal{X}}:x\neq y}r_{x}%
\Gamma_{xy}(r)\left[  e^{-\langle\alpha,e_{y}-e_{x}\rangle}-1\right]
\label{eq-Hfnsimp}%
\end{equation}
(we use $\boldsymbol{H}$ to distinguish from $H(x,p)$ as used in the section on
systems of Gibbs type). In the introduction of \cite{BuDuFiRa1} it was
discussed why the limit function $J^{\ast}$ may serve as a (local) Lyapunov
function for \eqref{EqLimitKolmogorov}. This suggests solutions of the
stationary PDE \eqref{eq:statpde} as candidates for a Lyapunov function. The
main result of this section makes this precise by proving that positive
definite subsolutions of \eqref{eq:statpde} give local Lyapunov functions for
\eqref{EqLimitKolmogorov}. To state the precise result we begin by recalling
the definition of a subsolution.

\begin{definition}
Let $\mathbb{D}$ be a relatively open subset of ${\mathcal{S}}$. A
${\mathcal{C}}^{1}$ function $J:\mathbb{D}\rightarrow\mathbb{R}$ is said to be
a \textbf{subsolution} of \eqref{eq:statpde} on $\mathbb{D}$ if
\begin{equation}
\boldsymbol{H}(r,-DJ(r))\geq0,\;\mbox{ for all }r\in\mathbb{D}.
\label{eq:statPDE}%
\end{equation}
Moreover, $J$ is said to be a solution to the PDE if (\ref{eq:statPDE}) holds
with equality.
\end{definition}

Recall the definition of positive definiteness given in Definition \ref{posdef}. Theorem \ref{prop:lyap} below says that a positive definite subsolution  of the PDE \eqref{eq:statpde} is a local
Lyapunov function.

\begin{theorem}
\label{prop:lyap} Suppose Condition \ref{cond:lip} holds.
Let $\pi^{\ast}\in{\mathcal{S}}^{\circ}$ be a fixed point of
\eqref{EqLimitKolmogorov}, and let $\mathbb{D}$ be a relatively open subset of
${\mathcal{S}}$ that contains $\pi^{\ast}$. Let $J:\mathbb{D}\rightarrow
\mathbb{R}$ be a ${\mathcal{C}}^{1}$ positive definite function that is a
subsolution of \eqref{eq:statpde} on $\mathbb{D}$. Then $J$ is a local
Lyapunov function for (\ref{EqLimitKolmogorov}) associated with $(\mathbb{D},
\pi^{\ast})$.
\end{theorem}

Before proceeding with the proof of the theorem we note some basic properties
of the function $\boldsymbol{H}$ introduced in \eqref{eq-Hfnsimp}.

\subsection{Properties of $\boldsymbol{H}$}

\label{sec:stpde}
%Recall the definition of $\boldsymbol{H}$ from \eqref{eq-Hfnsimp}.

\begin{lemma}
\label{lem:concav} Fix $r\in\mathcal{S}$ and $\alpha,\tilde{\alpha}%
\in\mathbb{R}^{d}$.

\begin{enumerate}
\item[(a)] If $\tilde{\alpha}-\alpha=c\boldsymbol{1}$ for some $c\in
\mathbb{R}$, then $\boldsymbol{H}(r,\alpha)=\boldsymbol{H}(r,\tilde{\alpha})$.

\item[(b)] $\boldsymbol{H}(r,\cdot)$ is smooth and concave on $\mathbb{R}^{d}$.

\item[(c)] Suppose that for each $p\in\mathcal{S}$, $\Gamma(p)$ is the rate
matrix of an ergodic Markov chain. Then given $r\in\mathcal{S}^{\circ}$,
$\tilde{\alpha}-\alpha\in\mathbb{R}^{d}\setminus\{c\boldsymbol{1}%
:c\in\mathbb{R}\}$ and any $\rho\in(0,1)$,
\[
\boldsymbol{H}(r,\rho\tilde{\alpha}+(1-\rho)\alpha)>\rho\boldsymbol{H}%
(r,\tilde{\alpha})+(1-\rho)\boldsymbol{H}(r,\alpha).
\]

\end{enumerate}
\end{lemma}

\begin{proof}
The definition of $\boldsymbol{H}$ in \eqref{eq-Hfnsimp} immediately implies
part (a), and (b) follows since the map $\alpha\mapsto e^{-\langle
\alpha,v\rangle}-1$ is smooth and convex for any vector $v\in\mathbb{R}^{d}$
and $r_{x}\Gamma_{xy}(r)\geq0$ for all $x\neq y$, $r\in{\mathcal{S}}$. To
prove (c), fix $r\in{\mathcal{S}}^{\circ}$ and $\alpha,\tilde{\alpha}%
\in\mathbb{R}^{d}$ such that $w\doteq\tilde{\alpha}-\alpha\not \in
\{c\boldsymbol{1}:c\in\mathbb{R}\}$. Then there exist $\bar{x},\bar{y}%
\in\{1,\ldots,d\}$ such that $w_{\bar{x}}\neq w_{\bar{y}}$. Due to the
smoothness and concavity of $\boldsymbol{H}(r,\cdot)$, it suffices to show
that
\begin{equation}
\frac{d^{2}}{d\rho^{2}}\boldsymbol{H}(r,\rho\tilde{\alpha}+(1-\rho
)\alpha)=\frac{d^{2}}{d\rho^{2}}\boldsymbol{H}(r,\rho w+\alpha)<0.
\label{concave}%
\end{equation}
Note that
\begin{align*}
\frac{d^{2}}{d\rho^{2}}\boldsymbol{H}(r,\rho w+\alpha)  &  =\frac{d^{2}}%
{d\rho^{2}}\left[  -\sum_{x,y\in{\mathcal{X}}:x\neq y}\left(  e^{\left\langle
-\alpha-\rho w,e_{y}-e_{x}\right\rangle }-1\right)  r_{x}\Gamma_{xy}(r)\right]
\\
&  =-\sum_{x,y\in{\mathcal{X}}:x\neq y}\left(  w_{y}-w_{x}\right)
^{2}e^{\left\langle -\alpha-\rho w,e_{y}-e_{x}\right\rangle }r_{x}\Gamma
_{xy}(r).
\end{align*}
Since $\Gamma$ is ergodic there is a sequence of distinct states $\bar
{x}=x_{1},x_{2},\ldots,x_{j}=\bar{y}$ such that $\Gamma_{x_{i}x_{i+1}}(r)>0$.
Also, since $w_{\bar{x}}\neq w_{\bar{y}}$, for some $i$ we have $w_{x_{i}}\neq
w_{x_{i+1}}$, and so (\ref{concave}) follows.
\end{proof}

\medskip For $z\in\lbrack0,\infty)$ let
\[
\ell(z)=z\log z-z+1.
\]
Given $r\in{\mathcal{S}}$, $\beta\in\mathbb{R}^{d}$, define
\begin{align}
&  \boldsymbol{L}(r,\beta)\nonumber\\
&  \doteq\inf_{u_{xy}\in\mathbb{R}_{+},y\neq x}\left[  \sum_{x,y\in
{\mathcal{X}}:x\neq y}r_{x}\Gamma_{xy}(r)\ell\left(  \frac{u_{xy}}{r_{x}%
\Gamma_{xy}(r)}\right)  :\sum_{x,y\in{\mathcal{X}}:x\neq y}(e_{y}-e_{x}%
)u_{xy}=\beta\right]  .\label{eq:lfnr}%
\end{align}
The following lemma establishes duality relations between $\boldsymbol{L}$ and
$\boldsymbol{H}$.

\begin{lemma}
\label{lem:dual} Fix $r\in{\mathcal{S}}$. For $\beta\in\mathbb{R}^{d}$
\begin{equation}
\boldsymbol{L}(r,\beta)=-\inf_{\alpha\in\mathbb{R}^{d}}\left[  \langle
\alpha,\beta\rangle-\boldsymbol{H}(r,\alpha)\right]  , \label{HtoL}%
\end{equation}
and for $\alpha\in\mathbb{R}^{d}$%
\begin{equation}
\boldsymbol{H}(r,\alpha)=\inf_{\beta\in\mathbb{R}^{d}}\left[  \langle
\alpha,\beta\rangle+\boldsymbol{L}(r,\beta)\right]  . \label{LtoH}%
\end{equation}

\end{lemma}

\begin{proof}
Fix $r\in{\mathcal{S}}$, and define $\tilde{\boldsymbol{H}}(r,\alpha
)=-\boldsymbol{H}(r,-\alpha)$ for $\alpha\in\mathbb{R}^{d}$. Let
${\mathcal{V}}\doteq\{e_{y}-e_{x}:x,y\in{\mathcal{X}},x\neq y\}$. Then note
that for $\alpha\in\mathbb{R}^{d}$,
\begin{equation}
\tilde{\boldsymbol{H}}(r,\alpha)=\sum_{v\in{\mathcal{V}}}h_{v}(r,\alpha
),\label{tilde-H}%
\end{equation}
where for $v=e_{y}-e_{x}\in{\mathcal{V}}$,
\[
h_{v}(r,\alpha)=\left[  e^{\langle\alpha,v\rangle}-1\right]  r_{x}\Gamma
_{xy}(r),\quad\alpha\in\mathbb{R}^{d}.
\]
For $v\in{\mathcal{V}}$, let $\ell_{v}(r,\cdot)$ be the Legendre transform of
$h_{v}(r,\cdot)$:
\[
\ell_{v}(r,\beta)=\sup_{\alpha\in\mathbb{R}^{d}}\left[  \langle\alpha
,\beta\rangle-h_{v}(r,\alpha)\right]  ,\quad\beta\in\mathbb{R}^{d}.
\]
Then with $\Lambda_{v}(r)\doteq r_{x}\Gamma_{xy}(r)$ when $v=e_{y}-e_{x}$,%
\[
\ell_{v}(r,\beta)=\left\{
\begin{array}
[c]{ll}%
\Lambda_{v}(r)\ell\left(  \frac{\theta}{\Lambda_{v}(r)}\right)   &
\mbox{ if  }\beta=\theta v\mbox{ for some }\theta\geq0,\\
\infty & \mbox{ otherwise. }
\end{array}
\right.
\]
From (\ref{tilde-H}), it follows using standard properties of Legendre
transforms (see, e.g., Corollary D.4.2 of \cite{DupEllBook}) that the function
$\boldsymbol{L}(r,\cdot)$ defined in \eqref{eq:lfnr} is the Legendre transform
of the function $\tilde{\boldsymbol{H}}(r,\cdot)$, that is,
\[
\boldsymbol{L}(r,\beta)=\sup_{\alpha\in\mathbb{R}^{d}}\left[  \langle
\alpha,\beta\rangle-\tilde{\boldsymbol{H}}(r,\alpha)\right]  ,
\]
which is easily seen to be equivalent to (\ref{HtoL}). Finally, since
$\tilde{\boldsymbol{H}}$ is convex and continuous by Lemma \ref{lem:concav}%
(b), the duality property of Legendre transforms shows that
\[
\tilde{\boldsymbol{H}}(r,\alpha)=\sup_{\beta\in\mathbb{R}^{d}}\left[
\langle\alpha,\beta\rangle-\boldsymbol{L}(r,\beta)\right]  .
\]
This is clearly equivalent to the relation (\ref{LtoH}), and so the proof is complete.
\end{proof}

We now return to the proof of Theorem \ref{prop:lyap}.

\subsection{Proof of Theorem \ref{prop:lyap}}

\label{subs-pdelyap}

We begin by noting that, for any $r\in\mathcal{S}$, $\beta(r)=r\Gamma(r)$
satisfies
\begin{equation}
\boldsymbol{H}(r,0)=\inf_{\beta\in\mathbb{R}^{d}}\left[  \boldsymbol{L}%
(r,\beta)\right]  =\boldsymbol{L}(r,\beta(r))=0. \label{dual-1}%
\end{equation}
By (\ref{HtoL}), for any $\alpha\in\mathbb{R}^{d}$
\[
\boldsymbol{H}(r,\alpha)\leq\boldsymbol{L}(r,\beta(r))+\langle\alpha
,\beta(r)\rangle.
\]
We next prove for any $\alpha\neq0$, $\alpha\in\mathcal{H}_{0}$, that
\begin{equation}
\boldsymbol{H}(r,\alpha)<\boldsymbol{L}(r,\beta(r))+\langle\alpha
,\beta(r)\rangle. \label{stricineq}%
\end{equation}
We argue via contradiction, and thus assume that (\ref{stricineq}) holds with
equality. Note that we must have $\alpha_{x}\neq\alpha_{y}$ for some
$x,y\in\mathcal{X}$, since otherwise $\alpha\in\mathcal{H}_{0}$ implies%
\[
0=\alpha\cdot\boldsymbol{1}=d\alpha_{x},\;x\in\mathcal{X},
\]
and then $\alpha=0$. Since $\alpha_{x}\neq\alpha_{y}$ for some $x,y$, by Lemma
\ref{lem:concav} (c) it follows that $\boldsymbol{H}(r,\rho\alpha
)>\rho\boldsymbol{H}(r,\alpha)$ for $\rho\in(0,1)$, and thus
\[
\boldsymbol{H}(r,\rho\alpha)-\langle\rho\alpha,\beta(r)\rangle>\rho
\boldsymbol{L}(r,\beta(r)).
\]
Then from \eqref{dual-1}
\begin{align*}
0  &  =\rho\boldsymbol{L}(r,\beta(r))\\
&  <\boldsymbol{H}(r,\rho\alpha)-\langle\rho\alpha,\beta(r)\rangle\\
&  \leq-\inf_{\alpha\in\mathbb{R}^{d}}\left[  \langle\alpha,\beta
(r)\rangle-\boldsymbol{H}(r,\alpha)\right] \\
&  =\boldsymbol{L}(r,\beta(r))\\
&  =0,
\end{align*}
which is a contradiction. This proves (\ref{stricineq}).

Recall that by assumption $J$ is positive definite, and thus in particular
$DJ(r)\neq0$ whenever $r\neq\pi^{\ast}$. Applying \eqref{stricineq} to
$\alpha=-DJ(r)$ (recall $DJ(r)\in\mathcal{H}_{0}$), where $r\neq\pi^{\ast}$,
we get
\[
0\leq\boldsymbol{H}(r,-DJ(r))<-\langle DJ(r),\beta(r)\rangle+\boldsymbol{L}%
(r,\beta(r))=-\langle DJ(r),\beta(r)\rangle,
\]
where the first equality is a consequence of the fact that $J$ is a
subsolution of \eqref{eq:statpde} on $\mathbb{D}$, while the last equality
follows on noting that $\boldsymbol{L}(r,\beta(r))=0$. Thus $\langle
DJ(r),\beta(r)\rangle<0$ whenever $r\neq\pi^{\ast}$. Finally, note that
\[
\frac{d}{dt}J(p(t))=\langle DJ(p(t)),p(t)\Gamma(p(t))\rangle=\langle
DJ(p(t)),\beta(p(t))\rangle<0,
\]
for all $0\leq t<\tau$ such that $p(t)\neq\pi^{\ast}$, where $\tau=\inf
\{t\geq0:p(t)\in\mathbb{D}^{c}\}$, which establishes the claimed result.
$\rule{0.5em}{0.5em}$

\section{Locally Gibbs Systems}

\label{sec:examples}

The PDE characterization of Section \ref{subs-timedeppde} gives a recipe for
constructing local Lyapunov functions for \eqref{EqLimitKolmogorov}. Although
in general explicit solutions of \eqref{eq:statpde} are not available, there
is an important class of nonlinear Markov processes introduced below for which
solutions to the PDE (\ref{eq:statpde}) can be constructed explicitly, and
which generalizes the class of Gibbs systems.

\begin{definition}
\label{def-lGibbs} A family of transition rate matrices $\{\Gamma
(r)\}_{r\in{\mathcal{S}}}$ on ${\mathcal{X}}$ is said to be \textbf{locally
Gibbs} if the following two properties hold:

\begin{enumerate}
\item[(a)] for each $r \in{\mathcal{S}}$, $\Gamma(r)$ is the rate matrix of an
ergodic Markov chain on ${\mathcal{X}}$, whose stationary distribution we
denote by $\pi(r)$;

\item[(b)] there exists a ${\mathcal{C}}^{1}$ function $U$ on ${\mathcal{S}}$
such that for every $x,y\in{\mathcal{X}}$, $x\neq y$,
\begin{equation}
\frac{\pi(r)_{y}}{\pi(r)_{x}}=\exp\left(  -D_{e_{y}-e_{x}}U(r)\right)  ,
\label{cond-lGibbs}%
\end{equation}
where for $v\in\mathcal{H}_{0}$, $D_{v}U=\langle DU,v\rangle$.
\end{enumerate}

The function $U$ is referred to as the potential associated with the locally
Gibbs family.
\end{definition}

The following result gives a local Lyapunov function for the ODE
\eqref{EqLimitKolmogorov} associated with a locally Gibbs family.

\begin{theorem}
\label{lem-lGibbs0} Suppose the transition rate matrices $\{\Gamma
(r)\}_{r\in{\mathcal{S}}}$ are locally Gibbs with potential function $U$, and
let the function $J$ be defined by
\begin{equation}
J(r)=\sum_{x\in{\mathcal{X}}}r_{x}\log r_{x}+U(r),\quad r\in{\mathcal{S}}.
\label{def-IlGibbs}%
\end{equation}
Then $J$ is a solution to the PDE (\ref{eq:statpde}) on ${\mathcal{S}}^{\circ
}$. Suppose in addition that  Condition \ref{cond:lip} holds and $J$ is positive definite in a relatively open (in
${\mathcal{S}}$) neighborhood of any fixed point $\pi^{\ast}\in{\mathcal{S}%
}^{\circ}$ of the ODE (\ref{EqLimitKolmogorov}). Then $J$ is a local Lyapunov
function for the ODE (\ref{EqLimitKolmogorov}) associated with $\pi^{\ast}$
and the neighborhood.
\end{theorem}

\begin{proof}
Let $\{\Gamma(r)\}_{r\in{\mathcal{S}}}$, $U$ and $J$ be as in the statement of
the theorem. Let $\{\pi(r)\}_{r\in\mathcal{S}}$ be the corresponding
collection of stationary distributions on $\mathcal{S}$. Since $U$ is
${\mathcal{C}}^{1}$ on ${\mathcal{S}}^{\circ}$ by assumption, $J$ is clearly
also ${\mathcal{C}}^{1}$ on ${\mathcal{S}}^{\circ}$. We now show that $J$ is a
solution to the equation (\ref{eq:statPDE}). First note that, due to the
locally Gibbs condition (\ref{cond-lGibbs}), for $r\in{\mathcal{S}}$ and
$x,y\in{\mathcal{X}}$, $x\neq y$,
\[
e^{D_{e_{y}-e_{x}}J(r)}=\frac{r_{y}\pi(r)_{x}}{r_{x}\pi(r)_{y}}.
\]
Moreover, since $\pi(r)$ is the stationary distribution for the Markov chain
with transition rate matrix $\Gamma(r)$, for any $y\in{\mathcal{X}}$,
\[
\sum_{x\in{\mathcal{X}}:x\neq y}\pi(r)_{x}\Gamma_{xy}(r)=-\pi(r)_{y}%
\Gamma_{yy}(r).
\]
Therefore,
\begin{align*}
-\boldsymbol{H}(r,-DJ(r))  &  =\sum_{x,y\in{\mathcal{X}}:x\neq y}%
[e^{D_{e_{y}-e_{x}}J(r)}-1]r_{x}\Gamma_{xy}(r)\\
&  =\sum_{x,y\in{\mathcal{X}}:x\neq y}\frac{r_{y}\pi(r)_{x}-r_{x}\pi_{y}%
(r)}{\pi(r)_{y}}\Gamma_{xy}(r)\\
&  =\sum_{y\in{\mathcal{X}}}\frac{r_{y}}{\pi(r)_{y}}\sum_{x\in{\mathcal{X}%
}:x\neq y}\pi(r)_{x}\Gamma_{xy}(r)-\sum_{x\in{\mathcal{X}}}r_{x}\sum
_{y\in{\mathcal{X}}:y\neq x}\Gamma_{xy}(r)\\
&  =-\sum_{y\in{\mathcal{X}}}r_{y}\Gamma_{yy}(r)+\sum_{x\in{\mathcal{X}}}%
r_{x}\Gamma_{xx}(r)\\
&  =0.
\end{align*}
Thus $J$ solves the PDE (\ref{eq:statpde}) on ${\mathcal{S}}^{\circ}$. The
result now follows from Theorem \ref{prop:lyap}.
\end{proof}

\medskip In the rest of this section we will describe several examples that
correspond to locally Gibbs systems and also give an example that falls
outside this category, and show that for the latter setting in some cases, the
PDE (\ref{eq:statpde}) can still be used to construct local Lyapunov
functions. In what follows we will not discuss  the positive definiteness
property, and instead refer to a function that satisfies (\ref{eq:statPDE}) as
a \emph{candidate} Lyapunov function, with the understanding that if positive
definiteness is added such a function will in fact be a local Lyapunov function.

The rest of the section is organized as follows. Section \ref{ex-gibbs}
considers a class of models that are a slight extension of the Gibbs systems
studied in Section \ref{SectGibbs}. A particular case of locally Gibbs that
appears in several contexts is introduced and discussed in Section
\ref{subs-gGibbs}. Section \ref{subs-lGibbs1} presents two examples of
three-dimensional systems which in particular illustrate that Gibbs systems
are a strict subset of locally Gibbs systems. In Section \ref{sec:nearnbr} we
consider models with nearest neighbor transitions. Section \ref{subs-lGibbs2}
studies an example from telecommunications \cite{AntFriRobTib08} for which the
associated $N$-particle system has the feature of \textquotedblleft
simultaneous jumps.\textquotedblright\ We show that an explicit construction
of a Lyapunov function carried out in \cite{AntFriRobTib08} follows as a
special case of Theorem \ref{lem-lGibbs0}. All examples in Sections
\ref{ex-gibbs}-\ref{subs-lGibbs2} are locally Gibbs systems. Section
\ref{subs-nlgibbs} considers an example that demonstrates that the class of
models for which a non-trivial solution to the PDE (\ref{eq:statpde}) can be
obtained is strictly larger than that of locally Gibbs systems.

\subsection{Gibbs systems}

\label{ex-gibbs}

Recall the empirical measure functional $r^{N}:{\mathcal{X}}^{N}%
\rightarrow{\mathcal{S}}$ defined in (\ref{def-rn}). Also recall that
throughout we assume $r\mapsto\Gamma(r)$ is Lipschitz continuous. We now
introduce a class of models that slightly extend those studied in Section
\ref{SectGibbs} which, with an abuse of terminology, we once more refer to as
Gibbs systems.

\begin{definition}
\label{def:gibbsdef} Let $K:\mathcal{X}\times\mathbb{R}^{d}\rightarrow
\mathbb{R}$ be such that for each $x\in\mathcal{X}$, $K(x,\cdot)=K^{x}(\cdot)$
is a continuously differentiable function on $\mathbb{R}^{d}$. We say a family
of rate matrices $\{\Gamma(r)\}_{r\in\mathcal{S}}$ on $\mathcal{X}$ is
\textbf{Gibbs} with potential function $K$, if

\begin{enumerate}
\item[(a)] For each $r\in\mathcal{S}$, $\Gamma(r)$ is a rate matrix of an
ergodic Markov chain with state space $\mathcal{X}$.

\item[(b)] For each $N\in\mathbb{N}$ there exists a collection of rate
matrices $\{\Gamma^{N}(r)\}_{r\in\mathcal{S}}$ such that $\Gamma
^{N}\rightarrow\Gamma$ uniformly on $\mathcal{S}$ and the $N$-particle Markov
process $\boldsymbol{X}^{N}$, for which the jump rate of the transition
\[
\left(  x_{1},\ldots,x_{i-1},x_{i},x_{i+1},\ldots,x_{N}\right)  \mapsto\left(
x_{1},\ldots,x_{i-1},y,x_{i+1},\ldots,x_{N}\right)
\]
is $\Gamma_{x_{i}y}^{N}(r^{N}(\boldsymbol{x}))$, is reversible with unique
invariant measure
\begin{equation}
\boldsymbol{\pi}^{N}(\boldsymbol{x})=\frac{1}{Z_{N}}\exp\left(  -\sum
_{i=1}^{N}K(x_{i},r^{N}(\boldsymbol{x}))\right)  ,\qquad\boldsymbol{x}%
\in{\mathcal{X}}^{N}, \label{stat-pin}%
\end{equation}
where $Z_{N}$ is the normalization constant:
\[
Z_{N}=\sum_{\boldsymbol{x}\in{\mathcal{X}}^{N}}\exp\left(  -\sum_{i=1}%
^{N}K(x_{i},r^{N}(\boldsymbol{x}))\right)  .
\]

\end{enumerate}
\end{definition}

From Section 4 of \cite{BuDuFiRa1} it follows that the family of rate matrices
in equation \eqref{eq:eqlimgen} is Gibbs in the sense of Definition
\ref{def:gibbsdef}. Note however that Definition \ref{def:gibbsdef} allows for
more general forms of rate matrices than \eqref{eq:eqlimgen}.
%Several different forms of $\Gamma^N$ such as the Glauber dynamics or the Metropolis-Hastings dynamics, lead to Markov
%chains on ${\mathcal{X}}^{N}$ that are reversible with respect to the
%stationary distribution $\boldsymbol{\pi}^{N}$; cf. Section 5 of \cite{BuDuFiRa1}.
%I
%For example, Glauber dynamics
%for the $N$-particle system corresponds to the case when the family of rate
%matrices $\Gamma^{N}(\cdot) = \Gamma(\cdot)$, where for $r \in{\mathcal{S}}$,
%%
%\[
%\Gamma_{x,y} (r) = A (x,y) \exp\left\{  - D_{e_{y} - e_{x}} U (r) \right\} ,
%\quad\mbox{ if } x, y \in{\mathcal{X}}, x \neq y,
%\]
%where $\{A(x,y), x,y \in{\mathcal{X}}\}$ is a symmetric, irreducible $0$-$1$
%matrix with zero diagonal entries, and, as usual, $\Gamma_{x,x} (r) = -
%\sum_{y \in{\mathcal{X}}: y \neq x}  \Gamma_{x,y} (r)$ (see \cite{DupFis12}
%for examples of other dynamics).

The following lemma shows that a Gibbs system is locally Gibbs in the sense
of Definition \ref{def-lGibbs}. 

\begin{lemma}
\label{lem-gibbslgibbs} If $\{\Gamma(r)\}_{r\in\mathcal{S}}$ is Gibbs with
some potential $K$, then it is locally Gibbs with potential $U(r)=\sum
_{z\in\mathcal{X}}K^{z}(r)r_{z}$.
\end{lemma}

\begin{proof}
Since $\boldsymbol{X}^{N}$ is reversible, the following detailed balance
condition on ${\mathcal{X}}^{N}$ must hold:
\begin{equation}
\boldsymbol{\pi}^{N}({\mathbf{x}})\Gamma_{x_{j}y}(r^{N}({\mathbf{x}%
}))=\boldsymbol{\pi}^{N}\left(  T_{y}^{j}{\mathbf{x}}\right)  \Gamma_{yx_{j}%
}(r^{N}(T_{y}^{j}{\mathbf{x}})) \label{balance}%
\end{equation}
for every ${\mathbf{x}}\in{\mathcal{X}}^{N}$, $y\in{\mathcal{X}}$ and
$j\in\{1,\ldots,N\}$, where $T_{y}^{j}{\mathbf{x}}$ has $j$th coordinate value
equal to $y$, and all other coordinates having values identical to those of
${\mathbf{x}}$. Since $r^{N}(T_{y}^{j}{\mathbf{x}})=r^{N}({\mathbf{x}}%
)+\frac{1}{N}(e_{y}-e_{x_{j}})$, by (\ref{stat-pin}) and (\ref{balance}), it
follows that for ${\mathbf{x}}\in{\mathcal{X}}^{N}$,
\begin{align}
&  \exp\left[  -\sum_{i=1}^{N}K(x_{i},r^{N}({\mathbf{x}}))+\sum_{i=1}%
^{N}K\left(  x_{i},r^{N}({\mathbf{x}})+\frac{1}{N}(e_{y}-e_{x_{j}})\right)
\right. \nonumber\\
&  \quad\left.  +K\left(  y,r^{N}(T_{y}^{j}{\mathbf{x}}))-K(x_{j}%
,r^{N}({\mathbf{x}})\right)  \rule{0pt}{22pt}\right]  \Gamma_{x_{j}y}%
^{N}(r^{N}({\mathbf{x}}))=\Gamma_{yx_{j}}^{N}(r^{N}(T_{y}^{j}{\mathbf{x}})).
\label{eq:detbalpre}%
\end{align}
Fix $x,y\in\mathcal{X}$, $x\neq y$ and $j\in\mathbb{N}$. Given $r\in
\mathcal{S}$, let $\{x_{1},x_{2},\cdots\}$ be a sequence in $\mathcal{X}$ such
that $x_{j}=x$ and with $\mathbf{x}^{N}=(x_{1},\ldots,x_{N})$, $r^{N}%
({\mathbf{x}}^{N})\rightarrow r$ as $N\rightarrow\infty$. Since $K$ is
continuously differentiable, as $N\rightarrow\infty$%
\begin{align*}
&  \sum_{i=1}^{N}K\left(  x_{i},r^{N}({\mathbf{x}}^{N})+\frac{1}{N}%
(e_{y}-e_{x})\right)  -\sum_{i=1}^{N}K(x_{i},r^{N}({\mathbf{x}}^{N}))\\
&  \quad=N\sum_{z\in\mathcal{X}}\left(  K\left(  z,r^{N}({\mathbf{x}}%
^{N})+\frac{1}{N}(e_{y}-e_{x})\right)  -K(z,r^{N}({\mathbf{x}}^{N}))\right)
r_{z}^{N}({\mathbf{x}}^{N})\\
&  \quad\rightarrow\sum_{z\in\mathcal{X}}\left(  \frac{\partial}{\partial
r_{y}}K(z,r)-\frac{\partial}{\partial r_{x}}K(z,r)\right)  r_{z}.
\end{align*}
Now, for $r\in{\mathcal{S}}$, define
\[
\pi(r)_{x}\doteq\frac{1}{Z(r)}\exp\left(  -H^{x}(r)\right)  ,\quad
r\in{\mathcal{X}},
\]
where $H$ was defined in \eqref{def-H} and $Z(r)$ is a normalization constant
to make $\pi(r)$ a probability measure. By sending $N\rightarrow\infty$ in
\eqref{eq:detbalpre} (with ${\mathbf{x}}$ replaced by ${\mathbf{x}}^{N}$),
using the uniform convergence of $\Gamma^{N}$ to $\Gamma$ and the fact that
$K$ is ${\mathcal{C}}^{1}$, we have
\begin{equation}
\frac{\pi(r)_{x}}{\pi(r)_{y}}\Gamma_{xy}(r)=\Gamma_{yx}(r),\quad
x,y\in{\mathcal{X}},x\neq y. \label{eq:newdetbal}%
\end{equation}
This shows that $\pi(r)$ is the stationary distribution for the rate matrix
$\Gamma(r)$, and thus verifies condition (a) of Definition \ref{def-lGibbs}.
Condition (b) holds because $U$ is ${\mathcal{C}}^{1}$ due to the assumptions
on $K$, and \eqref{cond-lGibbs} is verified by combining the last display with
the fact that $U(r)=\sum_{z\in\mathcal{X}}K^{z}(r)r_{z}$ and
(\ref{ExLimitStationary}) imply $-\langle DU(r),e_{y}-e_{x}\rangle=\log \pi(r)_{y}-\log \pi(r)_{x}$, $x,y\in{\mathcal{X}},x\neq y$. Thus, the
family $\{\Gamma(r)\}_{r\in{\mathcal{S}}}$ is locally Gibbs with potential $U$.
\end{proof}

%\begin{equation}
%\exp\Big(\langle DU(r),e_{y}-e_{x}\rangle\Big)\Gamma_{xy}(r)=\Gamma_{yx}(r).\label{eq:eqaldetbal}
%\end{equation}
%Since $r\in\mathcal{S}$ and $x,y\in\mathcal{X}$ are arbitrary we have that
%parts (a) and (b) in Definition \ref{def-lGibbs} hold and so the family
%$\{\Gamma(r),r\in{\mathcal{S}}\}$ is locally Gibbs with potential $U$.
%\end{proof}

\medskip Given a Gibbs family of matrices $\{\Gamma(r)\}_{r\in{\mathcal{S}}}$, it follows from Lemma
\ref{lem-gibbslgibbs} that the function $J:{\mathcal{S}}\rightarrow
\lbrack0,\infty)$ defined in (\ref{def-IlGibbs}) solves the stationary PDE
\eqref{eq:statpde} and thus serves as a candidate Lyapunov function. Example
\ref{ExmplMultEquilibria} shows that in general multiple fixed points of the
forward equation \eqref{EqLimitKolmogorov} exist and that the function $J$ may
be positive definite in the sense of Definition 
\ref{posdef} for
some of the fixed points and not positive definite for others.

The \emph{locally Gibbs} condition is significantly weaker than the Gibbs
property. Indeed, it follows from \eqref{eq:newdetbal} that Gibbs systems
satisfy the detailed balance condition for their corresponding rate matrices,
but systems with the locally Gibbs property need not satisfy this property.
The simplest example is as follows. Let $\pi$ be the invariant distribution
for the ergodic rate matrix $\Gamma$, and assume that detailed balance does
\emph{not} hold, so that it cannot be a Gibbs family. However, it is still
locally Gibbs, with $U(r)=\left\langle r,v\right\rangle ,v_{x}=-\log\pi_{x}$.
Note that in this case, the proposed Lyapunov function $J(r)$ in Theorem
\ref{lem-lGibbs0} is just the relative entropy $R(r\left\Vert \pi\right.  )$.
Example \ref{ex:5.4} below will also illustrate this point.

\subsection{A class of locally Gibbs systems}

\label{subs-gGibbs} We now introduce a family of ergodic rate matrices
$\{\Gamma(r)\}_{r\in{\mathcal{S}}}$ that describe limits of particle systems
whose dynamics need not be reversible for each $N$, (and hence may not be
Gibbs systems), but nevertheless have a structure that has some similarities
with Gibbs systems. We show that they are locally Gibbs, and then give two
concrete examples where they arise.

\begin{condition}
\label{def-gGibbs} For a family of transition rate matrices $\{\Gamma
(r)\}_{r\in{\mathcal{S}}}$ on ${\mathcal{X}}$ the following two properties hold.

\begin{enumerate}

\item[(a)] For each $r\in{\mathcal{S}}$, $\Gamma(r)$ is the rate matrix of an
ergodic Markov chain on ${\mathcal{X}}$ with stationary distribution $\pi
(r)$.
%and the map $r \mapsto \Gamma(r)$ is Lipschitz continuous;

\item[(b)] There exist $R:{\mathcal{X}}\times\lbrack0,1]\rightarrow\mathbb{R}$
and $K:{\mathcal{X}}\times\mathbb{R}^{d}\rightarrow\mathbb{R}$ such that for
each $x\in{\mathcal{X}}$, $R^{x}(\cdot)=R(x,\cdot)$ is a continuous function
and $K^{x}(\cdot)=K(x,\cdot)$ is a $\mathcal{C}^{1}$ function on
${\mathcal{S}}$, and such that for each $r\in{\mathcal{S}}$, $\pi(r)$ has the
form
\begin{equation}
\pi(r)_{x}=\frac{\exp[-H(x,r)-R(x,r_{x})]}{Z(r)},\quad x\in{\mathcal{X}},
\label{stat-gGibbs}%
\end{equation}
where $H$ is defined in terms of $K$ as in \eqref{def-H},
%\[  H^x(r) = H(x,r) = K^x (r) + \sum_{z \in {\mathcal X}} \left(
%\frac{\partial}{\partial r_x}  K^z (r) \right) r_z,
%\]
and $Z(r)$ is, as usual, the normalization constant
\[
Z(r)=\sum_{x\in{\mathcal{X}}}\exp[-H(x,r)-R(x,r_{x})].
\]

\end{enumerate}

%We will refer to $R$ and $K$ as the weight
%function and potential, respectively, associated with the
%generalized Gibbs family.

\end{condition}

Note that the Gibbs systems from Section \ref{ex-gibbs} satisfy Condition \ref{def-gGibbs} with
$R(x,t)=1$, $(x,t)\in \mathcal{X}\times [0,1]$. 
% In the particular case when $R$ is constant, $\{\Gamma(r)\}_{r\in{\mathcal{S}%
% }}$ is a Gibbs family in the sense of Section \ref{SectGibbs}.

\begin{remark}
\label{rem-gGibbs} \emph{ Given a family $\{\Gamma(r)\}_{r\in{\mathcal{S}}}$,
suppose there exist $R$ and $K$ as in Condition \ref{def-gGibbs} such that for
every $x,y\in{\mathcal{S}}$, $x\neq y$,}
\begin{equation}
\exp[H(y,r)+R(y,r_{y})]\Gamma_{xy}(r)=\exp[H(x,r)+R(x,r_{x})]\Gamma
_{yx}(r).\label{cond-gGibbs}%
\end{equation}
\emph{Then, for fixed $r \in \mathcal{S}$,  $\Gamma(r)$ satisfies the detailed balance conditions with
stationary distribution $\pi(r)$ given by \eqref{stat-gGibbs}, and so
$\{\Gamma(r)\}_{r\in{\mathcal{S}}}$ satisfies Condition \ref{def-gGibbs}.}
\end{remark}

%thus justifying the nomenclature of ``generalized Gibbs''.

\begin{lemma}
\label{lem-gGibbs} Let $\{\Gamma(r)\}_{r\in{\mathcal{S}}}$ satisfy Condition
\ref{def-gGibbs}. Then $\{\Gamma(r)\}_{r\in{\mathcal{S}}}$ is locally Gibbs
with potential
\begin{equation}
U(r)=\sum_{z\in{\mathcal{X}}}\left[  \int_{0}^{r_{z}}R(z,w)\,dw+K(z,r)r_{z}%
\right]  . \label{pot-gGibbs}%
\end{equation}

\end{lemma}

\begin{proof}
First, note that the conditions on $R$ and $K$ ensure that $U$ is a
${\mathcal{C}}^{1}$ function on ${\mathcal{S}}$. Thus, it suffices to verify
equation \eqref{cond-lGibbs} of Definition \ref{def-lGibbs}, namely to show
that for every $r\in{\mathcal{S}}$ and $x,y\in{\mathcal{X}}$,
\[
-\log\left(  \frac{\pi(r)_{y}}{\pi(r)_{x}}\right)  =D_{e_{y}-e_{x}}U(r).
\]
But this is a simple consequence of the identity
\[
\frac{\partial}{\partial r_{x}}\left[  \sum_{z\in{\mathcal{X}}}\int_{0}%
^{r_{z}}R(z,w)dw\right]  =R(x,r_{x}),\quad x\in{\mathcal{X}},
\]
the fact that from \eqref{def-H} we have
\[
\frac{\partial}{\partial r_{x}}\left(  \sum_{z\in{\mathcal{X}}}K(z,r)r_{z}%
\right)  =H(x,r),\quad x\in{\mathcal{X}},%
\]
and the definitions of $\pi(r)$ and $U$ in \eqref{stat-gGibbs} and
\eqref{pot-gGibbs}, respectively.
\end{proof}

\medskip We now provide two classes of models that satisfy Condition
\ref{def-gGibbs}. The first class is a system with only nearest-neighbor jumps.

\begin{example}
\label{gennnb} \emph{ Let $a_{i}$, $i=1,2,\ldots,d-1,$ and $b_{i}$,
$i=2,3,\ldots,d,$ be continuous maps from $\mathcal{S}$ to $(0,\infty)$.
Suppose that for $r\in\mathcal{S}$, $\Gamma(r)$ is associated with a birth
death chain as follows:
\begin{align*}
\Gamma_{i,i+1}(r)  &  =a_{i}(r),\;\;i=1,2,\ldots,d-1,\\
\Gamma_{i,i-1}(r)  &  =b_{i}(r),\;\;i=2,3,\ldots,d,\\
\Gamma_{i,j}(r)  &  =0,\;\;\mbox{ for all other }i\neq j.
\end{align*}
As usual, set $\Gamma_{ii}(r)=-\sum_{j,j\neq i}\Gamma_{ij}(r)$, so that
$\Gamma(r)$ is a rate matrix. }

\emph{Denoting by $\pi(r)$ the stationary distribution associated with
$\Gamma(r)$, the $\{\pi(r)_{i}\}_{i=1,\ldots,d}$ satisfy
\[
\pi(r)_{j}a_{j}(r)=\pi(r)_{j+1}b_{j+1}(r),\;\;\mbox{ for all }j=1,2,\ldots
,d-1.
\]
The following is a sufficient condition for Condition \ref{def-gGibbs}.
Suppose that there are measurable functions $\psi_{i}:\mathcal{S}%
\rightarrow(0,\infty)$, $i=0,\ldots,d-1$, that are bounded away from $0$, and,
for $i=1,\ldots,d$, continuous functions $\phi_{i}:[0,1]\rightarrow(0,\infty
)$, such that}
\begin{equation}
a_{i}(r)=\psi_{i}(r)\phi_{i}(r_{i}),\;\;b_{i}(r)=\psi_{i-1}(r)\phi_{i}%
(r_{i}),\;\;i=1,\ldots,d,\;r\in\mathcal{S}.\label{eq:nnbcond}%
\end{equation}
\emph{ Then, for $j=1,\ldots,d-1$,}
\[
\frac{\pi(r)_{j}}{\pi(r)_{j+1}}=\frac{b_{j+1}(r)}{a_{j}(r)}=\frac{\psi
_{j}(r)\phi_{j+1}(r_{j+1})}{\psi_{j}(r)\phi_{j}(r_{j})}=\frac{\phi
_{j+1}(r_{j+1})}{\phi_{j}(r_{j})}.
\]
\emph{It follows that $\{\Gamma(r)\}_{r\in{\mathcal{S}}}$ satisfies Condition
\ref{def-gGibbs} with $R(i,\cdot)=\log\phi_{i}(\cdot),i=1,\ldots,d,$ and
$K\equiv0$. By Lemma \ref{lem-gGibbs}, it follows that $\{\Gamma
(r)\}_{r\in{\mathcal{S}}}$ is locally Gibbs with potential }
\[
U(r)=\sum_{j=1}^{d}\int_{0}^{r_{j}}\log\phi_{j}(w)dw,\;u\in\lbrack
0,1],r\in{\mathcal{S}}.
\]

\end{example}

\begin{example}
\label{gGibbs-metropolis} \emph{ This example can be viewed as a
generalization of the Glauber dynamics introduced in Section \ref{SectGibbs},
in which the rate at which a particle changes state can depend both on the
state of the particle and on the fraction of particles in that state. Suppose
that we are given $R$ and $K$ as in Definition \ref{def-gGibbs}, let $H$ be
defined as in \eqref{def-H}, and as in Section \ref{SectGibbs}, let
$\Psi:{\mathcal{X}}\times{\mathcal{X}}\times{\mathcal{S}}\rightarrow
\mathbb{R}$ be given by $\Psi(x,y,r)=H^{y}(r)-H^{x}(r)$, $x,y\in{\mathcal{X}}%
$, $r\in{\mathcal{S}}$, and let $(\alpha(x,y))_{x,y\in{\mathcal{X}}}$ be an
irreducible and symmetric matrix with diagonal entries equal to zero and
off-diagonal entries equal to either one or zero. Then, for $r\in{\mathcal{S}%
}$, define }
\[
\Gamma_{xy}(r)=\exp[-\Psi(x,y,r)-R(x,r_{x})]\alpha(x,y),\quad x,y\in
{\mathcal{X}},x\neq y.
\]
\emph{ Then the equality in \eqref{cond-gGibbs} clearly holds and thus
$\{\Gamma(r)\}_{r\in{\mathcal{S}}}$ satisfies Condition \ref{def-gGibbs} by
Remark \ref{rem-gGibbs}. }

\emph{ Next recall that Theorem \ref{lem-lGibbs0} shows that $J(r)=U(r)+\sum
_{x\in{\mathcal{X}}}r_{x}\log r_{x}$, with $U$ defined by \eqref{pot-gGibbs},
is a candidate Lyapunov function for the associated ODE
\eqref{EqLimitKolmogorov}. In the present example, consider the case when
there exists a common $R_{0}:[0,1]\rightarrow\mathbb{R}$ such that
$R(x,\cdot)=R_{0}(\cdot)$ for every $x\in{\mathcal{X}}$ and 
$K(x,r)=-\log(\nu_{x}R(x,\nu_{x}))$ for some probability measure $\nu\in{\mathcal{P}%
}({\mathcal{X}})$ (and hence $K(x,r)$ does not depend on }$r$\emph{). Setting
$\bar{R}_{0}(u)=ue^{R_{0}(u)}$ for $u\in\lbrack0,1]$, we then have (up to a
constant), }
\[
J(r)=\sum_{z\in{\mathcal{X}}}\int_{\nu_{z}}^{r_{z}} \log \left ( \frac{\bar{R}_0(w)}{\bar{R}_0(\nu_z)}\right) dw,
\]
\emph{which is non-negative if $\bar{R}_{0}$ is non-decreasing. An analog of
this functional for nonlinear diffusions living in an open subset $\Omega$ of
a Riemannian manifold appears in \cite{BerDeSGabJoLLan02}, where it was shown
to be equal to the large deviation functional of the so-called zero range
process. Moreover, under the condition that $\bar{R}_{0}$ is strictly
increasing, it was shown in \cite{BodLebMouVil13} that this functional (and a
slight generalization of it, where the logarithm in the integrand is replaced
by the derivative of a more general ${\mathcal{C}}^{2}$ function) serves as a
Lyapunov function for the associated nonlinear PDE. }
\end{example}

\subsection{Some three-dimensional examples}

\label{subs-lGibbs1}

Both  classes of locally Gibbs families studied so far had the property
that for each $r\in{\mathcal{S}}$, $\Gamma(r)$ is associated with a reversible
Markov chain for which the detailed balance condition \eqref{eq:newdetbal}
holds. This leads to two natural questions: (a) does every locally Gibbs family have the property
that $\Gamma(r)$ satisfies detailed balance for each $r\in{\mathcal{S}}$? 
(b) if $\Gamma(r)$ satisfies
detailed balance for each $r\in{\mathcal{S}}$, then does it correspond to a
locally Gibbs family? 
To
address these questions, we consider some simple three-dimensional examples;
specifically, Example \ref{ex:5.4} answers questions (a) in the negative while Example
 \ref{3dnnb} gives a partial answer to (b) by showing that $\Gamma(r)$ may satisfy
detailed balance for each $r\in{\mathcal{S}}$, but it may fail to be a
locally Gibbs family with any $\mathcal{C}^2$-potential.

%We start with a parameterized family of examples that are locally Gibbs for
%some values, but not locally Gibbs for others.

\begin{example}
\label{3dnnb} \emph{ Suppose that $d=3$, fix $a_{i},b_{j}>0$, $i=1,2,j=2,3$,
and let $B:{\mathcal{S}}\rightarrow(0,\infty)$ be some given function. For
$r\in\mathcal{S}$, consider the matrix
\begin{equation}
\left(
\begin{array}
[c]{ccc}%
-a_{1} & a_{1} & 0\\
b_{2}B(r) & -b_{2}B(r)-a_{2} & a_{2}\\
0 & b_{3} & -b_{3}%
\end{array}
\right)  ,\label{matrix}%
\end{equation}
Note that for any fixed $r\in{\mathcal{S}}$, the matrix (\ref{matrix})
corresponds to an ergodic transition matrix with stationary distribution
\begin{equation}
\pi(r)=\frac{1}{Z(r)}\left(  b_{2}b_{3}B(r),a_{1}b_{3},a_{1}a_{2}\right)
,\label{pip-statmat}%
\end{equation}
where $Z(r)$ is, as usual, the normalization constant. Note also that the
detailed balance condition \eqref{eq:newdetbal} is satisfied for this model.
However as see below, in general this is not locally Gibbs system with a
$\mathcal{C}^{2}$-potential. Let $\Gamma$ be the transition matrix in
(\ref{matrix}) when $B(r)$ is replaced by $1$, and let $r^{\ast}$ be the
associated stationary distribution. Fix $c=(c_{1},c_{2},c_{3})\in
\mathbb{R}^{3}$ and $\kappa\in\mathbb{R}$. Also, for $r\in{\mathcal{S}}$, let
$\Gamma(r)$ be the matrix in (\ref{matrix}) with
\begin{equation}
B(r)=e^{\kappa\left\langle r-r^{\ast},c\right\rangle }.\label{def-bp}%
\end{equation}
A simple calculation shows that if $\kappa\neq0$ and $c_{2}\neq c_{3}$, there
is no $C^{2}$ function $U$ that
satisfies the equality in \eqref{cond-lGibbs} for all $x\neq y$ and
$r\in\mathcal{S}$. Indeed, such a function should satisfy for suitable real
numbers $\alpha,\beta$
\begin{equation}
\langle DU,e_{r_{2}}-e_{r_{1}}\rangle=\kappa\left\langle c,r\right\rangle
+\alpha,\;\langle DU,e_{r_{2}}-e_{r_{3}}\rangle=\beta.\label{eq:cond848}%
\end{equation}
Taking second derivatives we see that
\[
\frac{\partial^{2}U}{\partial r_{2}^{2}}-\frac{\partial^{2}U}{\partial
r_{2}\partial r_{1}}=\kappa c_{2},\;\frac{\partial^{2}U}{\partial
r_{3}\partial r_{2}}-\frac{\partial^{2}U}{\partial r_{3}\partial r_{1}}=\kappa
c_{3}%
\]
and
\[
\frac{\partial^{2}U}{\partial r_{2}^{2}}-\frac{\partial^{2}U}{\partial
r_{2}\partial r_{3}}=0.
\]
Adding the last two equations and subtracting from the first, we have
\[
\frac{\partial}{\partial r_{1}}\left(  \frac{\partial U}{\partial r_{3}}%
-\frac{\partial U}{\partial r_{2}}\right)  =\kappa(c_{2}-c_{3}).
\]
However, from \eqref{eq:cond848} the left side equals $0$. Thus we must have
$c_{2}=c_{3}$ or $\kappa=0$. Consequently when $c_{2}\neq c_{3}$ and
$\kappa\neq0$ the model is not locally Gibbs. On the other hand, if
$c_{2}=c_{3}$, one can check that \eqref{cond-lGibbs} is satisfied with
\begin{equation}
U(r)=\kappa r_{1}\left\langle r^{\ast}-r,c\right\rangle +\log\left(
\frac{a_{1}a_{2}}{b_{2}b_{3}}\right)  r_{1}+\log\left(  \frac{a_{2}}{b_{3}%
}\right)  r_{2}+\frac{1}{2}\kappa r_{1}^{2}(c_{1}-c_{2}).\label{def-Umat}%
\end{equation}
Thus, when $c_{2}=c_{3}$, the model is a locally Gibbs system with potential
$U$.}
\end{example}

\begin{example}
\label{ex:5.4} \emph{Let $d=3$ and for $r\in\mathcal{S}$, define the rate
matrix $\tilde{\Gamma}(r)$ by%
\[
\tilde{\Gamma}(r)=\left(
\begin{array}
[c]{ccc}%
-2r_{2}r_{3} & r_{2}r_{3} & r_{2}r_{3}\\
2r_{1}r_{3} & -4r_{1}r_{3} & 2r_{1}r_{3}\\
0 & 3r_{1}r_{2} & -3r_{1}r_{2}%
\end{array}
\right)  .
\]
Then clearly $r\tilde{\Gamma}(r)=0$. }

\emph{As in Example \ref{3dnnb}, for $r\in{\mathcal{S}}$ let $\Gamma(r)$ be
the matrix defined by (\ref{matrix}) with $B(r)$ given by (\ref{def-bp}), and
let $\pi(r)$ be the associated stationary distribution specified in
(\ref{pip-statmat}). Suppose that $c_{2}=c_{3}=c$. It was noted in Example
\ref{3dnnb} that $\pi(r)$ satisfies (\ref{cond-lGibbs}), with $U$ as in
(\ref{def-Umat}). For $r\in\mathcal{S}$, define
\[
\bar{\Gamma}(r)=\tilde{\Gamma}(\pi(r)),
\]
which takes the explicit form
\[
\bar{\Gamma}(r)=Z^{-1}(r)\left(
\begin{array}
[c]{ccc}%
-2a_{1}^{2}a_{2}b_{3} & a_{1}^{2}a_{2}b_{3} & a_{1}^{2}a_{2}b_{3}\\
2a_{1}a_{2}b_{2}b_{3}B(r) & -4a_{1}a_{2}b_{2}b_{3}B(r) & 2a_{1}a_{2}b_{2}%
b_{3}B(r)\\
0 & 3a_{1}b_{2}b_{3}^{2}B(r) & -3a_{1}b_{2}b_{3}^{2}B(r)
\end{array}
\right)  .
\]
}

\emph{ }

\emph{Since $r\tilde{\Gamma}(r)=0,$ $\pi(r)\bar{\Gamma}(r)=0$. Thus for each
$r\in\mathcal{S}$, $\bar{\Gamma}(r)$ is the rate matrix of an ergodic Markov
chain with stationary distribution $\pi(r)$. Also, as noted earlier, $\pi(r)$
satisfies \eqref{cond-lGibbs}. Thus, the family $\{\bar{\Gamma}(r)\}_{r\in
\mathcal{S}}$ satisfies the local Gibbs property. However, note that the
detailed balance condition \eqref{eq:newdetbal}, which must hold for every
Gibbs model, fails. Indeed,
\[
\frac{\bar{\Gamma}_{12}(r)}{\bar{\Gamma}_{21}(r)}=\frac{a_{1}}{2b_{2}B(r)}%
\neq\frac{a_{1}}{b_{2}B(r)}=\frac{\pi(r)_{2}}{\pi(r)_{1}}.
\]
Thus $\{\bar{\Gamma}(r)\}_{r\in\mathcal{S}}$ is not Gibbs. }
\end{example}

\subsection{Systems with nearest neighbor jumps}

\label{sec:nearnbr}
%In this section we consider some nearest neighbor systems.
%In general such systems may not be locally Gibbs, but in Example \ref{gennnb}
%we will identify a sufficient condition for this property to hold. Example
%\ref{ex:mononnb} considers a more general setting where it is unclear if the
%model satisfies the locally Gibbs property. However there are slight
%variations of this model for which the locally Gibbs property does hold. One
%such case is considered in Example \ref{ex-3state}.
%\footnote{I have moved the first example to the Generalized Gibbs
%section since it was a special case of that.   I removed the second
%example
%as  suggested:  it is within iffalse's. We will have to modify the
%above section title and introduction depending on where we
%keep  these  examples here or move it here. }

The nearest neighbor model in Example \ref{gennnb} imposed certain symmetry
conditions (see \eqref{eq:nnbcond}) on the rate parameters. In the following
example we consider a more general family of near neighbor models with certain
monotonicity conditions on the rates.

\begin{example}
\label{ex-3state} \emph{Let ${\mathcal{X}}=\{1,\ldots,d\}$ and for
$r\in\mathcal{S}$, suppose there exist `cost' vectors $c^{i}\in\mathbb{R}^{d}%
$, $i=1,\ldots,d-1$, and continuous functions $a^{i}:\mathbb{R}%
\rightarrow(0,\infty)$ and $b^{i}:\mathbb{R}\rightarrow(0,\infty)$,
$i=1,\ldots,d,$ such that for every $r\in{\mathcal{S}}$, $\Gamma(r)$ is the
rate matrix of a birth-death chain that satisfies }
\begin{align*}
\Gamma_{i,i+1}(r)  &  =a^{i}(\langle r,c^{i}\rangle),\;\; & i  &
=1,2,\ldots,d-1,\\
\Gamma_{i+1,i}(r)  &  =b^{i+1}(\langle r,c^{i}\rangle),\;\; & i  &
=1,2,\ldots,d-1,\\
\Gamma_{i,j}(r)  &  =0,\;\; & \mbox{{\em for all other} }i  &  \neq j,
\end{align*}
\emph{and, as usual, set $\Gamma_{ii}(r)=-\sum_{i,j,j\neq i}\Gamma_{ij}(r)$,
so that $\Gamma(r)$ is a rate matrix. Let $\pi(r)$ denote the stationary
distribution of the chain with rate matrix $\Gamma(r)$. Since $\pi
(r)\Gamma(r)=0$, we have for $i=1,\ldots,d-1$, }
\begin{equation}
\frac{\pi(r)_{i+1}}{\pi(r)_{i}}=\frac{b^{i+1}(\langle r,c^{i}\rangle
)}{a^{i}(\langle r,c^{i}\rangle)}=\psi^{i}(\langle r,c^{i}\rangle),
\label{def-psi}%
\end{equation}
\emph{where $\psi^{i}(u)\doteq b^{i+1}(u)/a^{i}(u)$, $u\in\mathbb{R}$. }

\emph{Consider the specific case when the cost vectors have the form
$c^{j}=\sum_{k=j+1}^{d}e_{k}$, $j=1,\ldots,d-1$. Then for $i,j=1,\ldots,d-1$,
}
\begin{equation}
c_{i+1}^{j}-c_{i}^{j}=\left\{
\begin{array}
[c]{ll}%
1 & \mbox{ i = j, }\\
0 & \mbox{ {\em otherwise.} }
\end{array}
\right.  \label{cost-rel}%
\end{equation}
\emph{ Then we claim that $\{\Gamma(r)\}_{r\in{\mathcal{S}}}$ is a locally
Gibbs family with potential }
\[
U(r)=-\sum_{j=1}^{d-1}\int_{0}^{\langle r,c^{j}\rangle}\log\left(  \psi
^{j}(w)\right)  \,dw.
\]
\emph{ Indeed, for every $r\in{\mathcal{S}}$ and $i=1,\ldots,d-1$, using
\eqref{cost-rel} }
\[
-D_{e_{i+1}-e_{i}}U(r)=\sum_{j=1}^{d-1}\log\left(  \psi^{j}(\langle
r,c^{j}\rangle)\right)  (c_{i+1}^{j}-c_{i}^{j})=\log\left(  \psi^{i}(\langle
r,c^{i}\rangle)\right)  .
\]
\emph{Together with \eqref{def-psi} this shows that condition
\eqref{cond-lGibbs} is satisfied, and thus $\{\Gamma(r)\}_{r\in{\mathcal{S}}}$
is locally Gibbs. }
\end{example}

\subsection{Models with simultaneous jumps}

\label{subs-lGibbs2}

Weakly interacting particles systems with \textquotedblleft simultaneous
jumps\textquotedblright\ are described in \cite{AntFriRobTib08,DupRamWu12}. For our purposes
here we need only know that the nonlinear Markov process associated with such
models can also be interpreted as the limit process for an ordinary single
jump process, with an effective rate matrix that is defined in terms of
various rate matrices used in the definition of the original process. We
describe one model with simultaneous jumps that arises naturally in
telecommunications and which was studied in \cite{AntFriRobTib08}, and show
that the associated family of effective rate matrices is locally Gibbs.

%\subsubsection{A locally Gibbs simultaneous jumps model.}

%\begin{example}
%\label{ex-simjumps} \emph{
In this model there are $N$ nodes, each with capacity $C\in\mathbb{N}$, and
there are $M\in\mathbb{N}$ classes, each with parameters $\lambda_{m},\mu_{m}%
$, $\gamma_{m}>0$, and, in addition, a capacity requirement $A_{m}%
\in\mathbb{N}$. The state of node $i$ is the number of calls of each class
present at that node in an $N$-node network, and thus the state space takes
the form
\[
{\mathcal{X}}=\left\{  x\in\mathbb{Z}_{+}^{M}:\sum_{m=1}^{M}x_{m}A_{m}\leq
C\right\}  .
\]
Let $a_{m}(r)$ denote the average number of customers in class $m$ under the
distribution $r$:
\begin{equation}
a_{m}(r)\doteq\left(  \sum_{x\in{\mathcal{X}}}r_{x}x_{m}\right)  .
\label{am-mean}%
\end{equation}
It was shown in Theorem 1 of \cite{AntFriRobTib08} that the associated
sequence of empirical measures satisfies $\mu^{N}(\cdot)\Rightarrow p(\cdot)$,
as $N\rightarrow\infty$, where $p(\cdot)$ satisfies the ODE
(\ref{EqLimitKolmogorov}), with $\Gamma$ taking the following form: for
$r\in{\mathcal{S}}$ and $x,y\in{\mathcal{X}}$, $x\neq y$,
\begin{equation}
\Gamma_{x,y}(r)=\left\{
\begin{array}
[c]{ll}%
\lambda_{m}+\gamma_{m}a_{m}(r) & \mbox{ if }y=x+f_{m}\in{\mathcal{X}%
},m=1,\ldots,M,\\
x_{m}(\mu_{m}+\gamma_{m}) & \mbox{ if }y=x-f_{m}\in{\mathcal{X}}%
,m=1,\ldots,M,\\
0 & \mbox{ otherwise, }
\end{array}
\right.  \label{ant-gammaeff}%
\end{equation}
and, as usual, $\Gamma_{xx}(r)=-\sum_{y\in{\mathcal{X}},y\neq x}\Gamma
_{yx}(r)$. Moreover, it is easily verified (see Proposition 1 of
\cite{AntFriRobTib08}) that for each $r\in{\mathcal{S}}$, $\Gamma(r)$ is the
rate matrix of an ergodic Markov chain with stationary distribution $\pi(r)$
given by
\[
\pi(r)_{x}=\frac{1}{Z(r)}\frac{\prod_{m=1}^{M}\left(  \rho_{m}(r)\right)
^{x_{m}}}{\prod_{m=1}^{M}x_{m}!},\quad x=(x_{1},\ldots,x_{M})\in{\mathcal{X}%
},
\]
where $Z(r)$ is the normalization constant and $\rho:{\mathcal{S}}%
\rightarrow(0,\infty)^{M}$ is given by
\[
\rho_{m}(r)=\frac{\lambda_{m}+\gamma_{m}a_{m}(r)}{\mu_{m}+\gamma_{m}},\quad
m=1,\ldots,M.
\]
%It is clear that then $r^{\ast}$ is a fixed point of the ODE
%(\ref{EqLimitKolmogorov}) if and only if it solves the fixed point equation
%\begin{equation}
%r=\pi(r). \label{fp-antunes}%
%\end{equation}
It was shown in \cite{AntFriRobTib08} that when $M=1$ there is a unique fixed
point for the ODE (\ref{EqLimitKolmogorov}), but when $M=2$, $A_{1}=1$ and
$A_{2}=C$ for $C$ sufficiently large, there exist parameters $\lambda_{m}%
,\mu_{m}$ and $\gamma_{m}$ for which (\ref{EqLimitKolmogorov}) has multiple
fixed points.

\begin{lemma}
\label{lem-ant} The family of rate matrices $\{\Gamma(r)\}_{r\in{\mathcal{S}}%
}$ defined in (\ref{ant-gammaeff}) is locally Gibbs with potential
\[
U(r)=\sum_{m=1}^{M}\left[  \sum_{x\in\mathcal{X}}\left[  r_{x}\log (x_{m}!)-x_{m}r_{x}\log(\lambda_{m}+\mu_{m})\right]  +\int_{0}^{a_{m}(r)}%
\log\left(  \lambda_{m}+\gamma_{m}w\right)  \,dw\right]  .
\]

\end{lemma}

\begin{proof}
The function $U$ is clearly ${\mathcal{C}}^{1}$ on ${\mathcal{S}}$. For $r
\in{\mathcal{S}}$ and $x = (x_{1}, \ldots, x_{M}), y = (y_{1}, \ldots, y_{M})
\in{\mathcal{X}}$, we have
\begin{align*}
- \log\left(  \frac{\pi(r)_{x}}{\pi(r)_{y}} \right)   &  = \sum_{m=1}^{M}
\left[  (y_{m}-x_{m}) \log\left(  \rho_{m}( r)\right)  +\log(x_{m}!) -
\log(y_{m}!) \right] \\
&  = \sum_{m=1}^{M} \left[  (y_{m}- x_{m}) \log\left(  \lambda_{m} +
\gamma_{m} a_{m} (r) \right)  - (y_{m}-x_{m}) \log(\mu_{m}+ \gamma_{m})
\right. \\
&  \left.  \qquad\qquad+ \log(x_{m}!) - \log(y_{m}!) \right]  ,
\end{align*}
which is easily seen to coincide with $D_{e_{y} - e_{x}} U(r)$, thus
establishing that (\ref{cond-lGibbs}) is satisfied.
\end{proof}

\medskip By Theorem \ref{lem-lGibbs0}, it then follows under positive
definiteness that $J(r)=U(r)+\sum_{x\in{\mathcal{X}}}r_{x}\log r_{x}$ is a
local Lyapunov function for (\ref{EqLimitKolmogorov}). Using the definition
$a_{m}(r)=\sum_{x\in{\mathcal{X}}}x_{m}r_{x}$, it is easily seen that $J$
coincides with the Lyapunov function $g$ constructed in Proposition 4 of
\cite{AntFriRobTib08}.

\begin{remark}
\emph{Features of the last example are that the state space ${\mathcal{X}%
}\subset\mathbb{R}^{M}$ and the rate matrix depends on $r$ only through the
mean values $a_{m}(r),m=1,\ldots,M$. The example can be generalized slightly.
Indeed, consider a family of ergodic rate matrices $\{\Gamma(r)\}_{r\in
{\mathcal{S}}}$ on ${\mathcal{X}}\subset\mathbb{R}^{M}$, with the property
that for each $r\in{\mathcal{S}}$, $\Gamma(r)$ has a stationary distribution
$\pi(r)$ of the form }
\[
\pi(r)_{x}=\prod_{m=1}^{M}\left[  \Phi^{(m)}(a_{m}(r))\right]  ^{x_{m}}%
\exp(-H(x,r)),\quad x\in{\mathcal{X}},r\in{\mathcal{S}},
\]
\emph{where $H$ is the function defined in \eqref{def-H} for some
$K:{\mathcal{X}}\times\mathbb{R}^{d}\rightarrow\mathbb{R}$, $a_{m}$ is defined
by \eqref{am-mean} and for each $m=1,\ldots,M$, $\Phi^{(m)}:\mathbb{R}%
\rightarrow(0,\infty)$ is continuous. Then, using arguments exactly analogous
to those used previously in this section, one can show that $\{\Gamma
(r)\}_{r\in{\mathcal{S}}}$ is locally Gibbs with potential }
\[
U(r)=\sum_{z\in{\mathcal{X}}}\left[  \sum_{m=1}^{M}\int_{0}^{a_{m}(r)}%
\log\Phi^{(m)}(w)\,dw+K(z,r)r_{z}\right]  .
\]
\emph{The last example presented then coincides with the case $\Phi
_{m}(w)=\lambda_{m}+\gamma_{m}w$, $w\in\mathbb{R}$, and for $r\in{\mathcal{S}%
}$
\[
K(x,r)=H(x,r)=\sum_{m=1}^{M}\left(  x_{m}\log(\mu_{m}+\gamma_{m})+\log
(x_{m}!)\right)  .
\]
}
\end{remark}

\begin{remark}
\emph{ Another example of a model with simultaneous jumps is the model of
alternative routing in loss networks introduced in Gibbens, Hunt and Kelly
\cite{GibHunKel90}. It can be shown that the family of jump matrices
$\{\Gamma(r)\}_{r\in{\mathcal{S}}}$ associated with this model is not locally
Gibbs with any ${\mathcal{C}}^{2}$ potential $U$. This may explain why this
problem has withstood analysis for more than a decade. It is an interesting
open problem to see if the PDE characterization introduced here can be used to
construct Lyapunov functions for this model and related ones. }
\end{remark}

\subsection{A candidate Lyapunov function for a model that is not locally
Gibbs}

\label{subs-nlgibbs}

The example in this section demonstrates that the class of models for which
explicit non-zero solutions of \eqref{eq:statpde} can be found is larger than
that of locally Gibbs models. Let $d=3$, and for $r\in\mathcal{S}$ define the
rate matrix $\Gamma(r)$ by%
\[
\Gamma(r)=\left(
\begin{array}
[c]{ccc}%
-a_{1}(r) & a_{1}(r) & 0\\
b_{2}(r) & -(a_{2}(r)+b_{2}(r)) & a_{2}(r)\\
0 & b_{3}(r) & -b_{3}(r)
\end{array}
\right)  ,
\]
where $a_{1}$ and $a_{2}$ are measurable functions from $\mathcal{S}$ to
$(0,1)$ and $b_{2}$, $b_{3}$ are given as follows.
%such that $r_{1}(r)\vee r_{2}(r)\leq 1/4$.
Let $\psi:[0,1]\rightarrow(0,1)$ be a continuous function that is bounded away
from $0$. We set
\[
b_{2}(r)=\left(  1+(r_{2}-r_{3}\psi(r_{3}))a_{2}(r)\right)  a_{1}%
(r),\;b_{3}(r)=\psi(r_{3})a_{2}(r)\left(  1+(r_{2}-r_{1})a_{1}(r)\right)  .
\]
Note that for each $r\in\mathcal{S}$, $\Gamma(r)$ is an ergodic rate matrix
and the corresponding unique invariant measure $\pi(r)$ satisfies
\begin{align*}
\frac{\pi(r)_{1}}{\pi(r)_{2}}=  &  \frac{b_{2}(r)}{a_{1}(r)}=\left(
1+(r_{2}-r_{3}\psi(r_{3}))a_{2}(r)\right)  ,\\
\frac{\pi(r)_{2}}{\pi(r)_{3}}=  &  \frac{b_{3}(r)}{a_{2}(r)}=\psi
(r_{3})\left(  1+(r_{2}-r_{1})a_{1}(r)\right)  .
\end{align*}
Since $a_{1},a_{2}$ are arbitrary functions, there may be no $\mathcal{C}^1$ function $U$ for which equation \eqref{cond-lGibbs}
is satisfied, and so the family $\{\Gamma
(r)\}_{r\in{\mathcal{S}}}$ is not locally Gibbs in general.

Define
\[
U(r)\doteq\int_{0}^{r_{3}}\log\psi(x)dx,\;\;r\in\mathcal{S},
\]
and let $J$ be defined through \eqref{def-IlGibbs}. Then, as shown below, $J$
satisfies the PDE (\ref{eq:statpde}) on $\mathcal{S}^{\circ}$ and hence is a
candidate Lyapunov function. Indeed, note that
\[
D_{e_{y}-e_{x}}J(r)=\left\{
\begin{array}
[c]{cc}%
\log(\frac{r_{2}}{r_{1}}) & \mbox{ if }(y,x)=(2,1),\\
\  & \\
\log(\frac{r_{3}\psi(r_{3})}{r_{1}}) & \mbox{ if }(y,x)=(3,1),\\
\  & \\
\log(\frac{r_{3}\psi(r_{3})}{r_{2}}) & \mbox{ if }(y,x)=(3,2).
\end{array}
\right.
\]
Thus,
\begin{align*}
-\boldsymbol{H}(r,-DJ(r))  &  =(r_{2}-r_{1})a_{1}(r)+(r_{3}\psi(r_{3}%
)-r_{2})a_{2}(r)+(r_{1}-r_{2})b_{2}(r)\\
&  \quad+\frac{(r_{2}-r_{3}\psi(r_{3}))}{\psi(r_{3})}b_{3}(r)\\
&  =(r_{2}-r_{1})a_{1}(r)+(r_{3}\psi(r_{3})-r_{2})a_{2}(r)\\
&  \quad+(r_{1}-r_{2})\left(  1+(r_{2}-r_{3}\psi(r_{3}))a_{2}(r)\right)
a_{1}(r)\\
&  \quad+\frac{(r_{2}-r_{3}\psi(r_{3}))}{\psi(r_{3})}\psi(r_{3})a_{2}%
(r)\left(  1+(r_{2}-r_{1})a_{1}(r)\right) \\
&  =0.
\end{align*}

\noindent\textbf{Acknowledgments. } We would like to thank Vaios Laschos for
bringing the paper \cite{BodLebMouVil13} to our attention. \newline


\begin{thebibliography}{99}                                                                                               %


%\bibitem {AntFriRobTibConf06}N.~Antunes, C.~Fricker, P.~Robert, and D.~Tibi.
%\newblock Metastability of {CDMA} cellular systems. \newblock In
%\emph{Proceedings of {MOBICOM}}, Los Angeles, 2006.


\bibitem {AntFriRobTib08}N.~Antunes, C.~Fricker, P.~Robert, and D.~Tibi.
\newblock Stochastic networks with multiple stable points.
\newblock {\em Ann.\ Probab.}, 36(1):255--278, 2008.

%\bibitem {BenLeB08}M.~Benaim and J.Y. Le~Boudec. \newblock A class of mean
%field interaction models for computer and communication systems.
%\newblock {\em Performance Evaluation}, 65(11-12):823--838, 2008.


\bibitem {BerDeSGabJoLLan02}L.~Bertini, A. De Sole, D.~Gabrielli,
G.~Jona-Lasinio, and C.~Landim. \newblock Macroscopic fluctuation theory for
stationary non-equilibrium states. \newblock {\em J. Statist. Phys.}, 107,
3-4:635--675, 2002.

\bibitem {BodLebMouVil13}T.~Bodineau, J.L.~Lebowitz, C.~Mouhot and C.~Villani.
\newblock Lyapunov functions for boundary-driven nonlinear drift-diffusions.
\newblock {\em Nonlinearity}, 27, 9:2111--20132,  2014.

%\bibitem {BorMcdPro08}C.~Bordenave, D.~McDonald, and A.~Proutiere.
%\newblock Performance of random medium access control, an asymptotic approach.
%\newblock In \emph{Proc. {ACM} {S}igmetrics 2008}, pages 1--12, 2008.


%\bibitem {BorMcdPro12}C.~Bordenave, D.~McDonald, and A.~Proutiere.
%\newblock Asymptotic stability region of slotted-aloha.
%\newblock {\em IEEE Transactions on Information Theory}, 58(9):5841--5855, 2012.

\bibitem{BorSun}
V.S.~Borkar and R.~Sundaresan.
\newblock Asymptotics of the invariant measure in mean field models with jumps.
\newblock {\em Stochastic Systems}, 2:1--59,
2012.


\bibitem {BuDuFiRa1}A.~Budhiraja, P.~Dupuis, M.~Fischer and K.~Ramanan.
\newblock Limits of relative entropies associated with weakly interacting
particle systems. \newblock {\em Preprint.}

\bibitem {CaMcVi}J.A. Carrillo, R. J. McCann and C. Villani. \newblock Kinetic
equilibration rates for granular media and related equations: Entropy
dissipation and mass transportation estimates.
\newblock {\em Rev. Mat. Iberoamericana}, 19, 971--1018, 2003.

%\bibitem {daipraetalii09}P.~Dai~Pra, W.~J. Runggaldier, E.~Sartori, and
%M.~Tolotti. \newblock Large portfolio losses: {A} dynamic contagion model.
%\newblock \emph{Ann. Appl. Probab.}, 19\penalty0 (1):\penalty0 347--394, 2009.


%\bibitem {dawsonetalii05}D.~A. Dawson, J.~Tang, and Y.~Q. Zhao.
%\newblock Balancing queues by mean field interaction. \newblock \emph{Queueing
%Syst.}, 49:\penalty0 335--361, 2005.


%\bibitem {DemZeiBook}A.~Dembo and O.~Zeitouni.
%\newblock {\em Large Deviations Techniques and Applications}.
%\newblock Bartlett Publishers, Boston, MA, 1993.


\bibitem {DupEllBook}P.~Dupuis and R.~Ellis.
\newblock {\em A Weak Convergence Approach to the Theory of Large Deviations}.
\newblock John Wiley \& Sons, New York, 1997.

%\bibitem {DupFis12}P.~Dupuis and M.~Fischer. \newblock On the construction of
%Lyapunov functions for nonlinear Markov processes via relative entropy. \newblock Preprint.


\bibitem {DupRamWu12}P.~Dupuis, K.~Ramanan, and W.~Wu. \newblock Sample path
large deviation principle for mean field weakly interacting jump processes.
\newblock In preparation., 2012.

%\bibitem {frank01}T.~D. Frank. \newblock Lyapunov and free energy functionals
%of generalized {F}okker-{P}lanck equations. \newblock {\em Phys. Lett. A},
%290: 93--100, 2001.


%\bibitem {FreWen}M.I.~Freidlin and A.D.~Wentzell. \newblock Random
%Perturbations of Dynamical Systems. \newblock Springer, 1998.


\bibitem {GibHunKel90}R.J. Gibbens, P.J. Hunt, and F.P. Kelly.
\newblock {\em Bistability in communication networks}, pages 113--127.
\newblock Oxford Sci. Publ., Oxford University Press, New York, 1990.

%\bibitem {GomGraLeB12}J.~G\'{o}mez-Serrano, C.~Graham, and Y.~Le~Boudec.
%\newblock The bounded confidence model of opinion dynamics.
%\newblock {\em Math. Models Methods Appl. Sci.}, 22(2), 2012.


%\bibitem {gottlieb98}A.~D. Gottlieb. \newblock \emph{Markov transitions and
%the propagation of chaos}. \newblock PhD thesis, Lawrence Berkeley National
%Laboratory, 1998.


%\bibitem {graham92}C.~Graham. \newblock Mc{K}ean-{V}lasov {I}t{\^o}-{S}korohod
%equations, and nonlinear diffusions with discrete jump sets.
%\newblock {\em Stochastic Processes Appl.}, 40 (1): 69--82, 1992.


%\bibitem {Gra00}C.~Graham. \newblock Chaoticity on path space for a queueing
%network with selection of the shortest queue amongst several.
%\newblock {\em J.\ Appl.\ Probab.}, 37(1):198--211, 2000.


%\bibitem {GraRob10}C.~Graham and P.~Robert. \newblock A multi-class mean-field
%model with graph structure for tcp flows. \newblock In Heidelberg Springer,
%editor, \emph{Progress in Industrial mathematics at {ECMI} 2008}, volume~15,
%pages 125--131, 2010.


\bibitem {hollander00}F.\ den~Hollander. \newblock {\em Large Deviations}.
\newblock American Mathematical Society, Providence (Rhode Island), 2000.

\bibitem{Leo}
C.~L\'{e}onard.
\newblock  Large deviations for long range interacting particle systems with jumps. 
\newblock {\em Ann. Inst. H. Poincar\'{e} Probab. Statist.}, 31(2):289--323, 1995.

%\bibitem {HuaMalCai06}M.Y. Huang, R.P. Malham\'{e}, and P.E. Caines.
%\newblock Large population stochastic dynamic games: closed loop {M}%
%c{K}ean-{V}lasov systems and the {N}ash certainty equivalence principle.
%\newblock {\em Communications in Information and Systems}, 6(3):221--252,
%2006. \newblock Special issue in honour of the 65th birthday of Tyrone Duncan.


%\bibitem {KolBook11}V.N. Kolokoltsov.
%\newblock {\em Nonlinear Levy and Nonlinear Feller processes}. \newblock De
%Gruyter, 2011.


%\bibitem {Kol12}V.N. Kolokoltsov. \newblock Nonlinear markov games on a finite
%state space (mean-field and binary interactions).
%\newblock {\em International Journal of Statistics and Probability},
%1(1):77--91, 2012.


%\bibitem {Kur70}T.~Kurtz. \newblock Solutions of ordinary differential
%equations as limits of pure jump markov processes.
%\newblock {\em Jour. Appl.\ Probab.}, 7(1):49--58, 1970.


%\bibitem {martinelli99}F.~Martinelli. \newblock Lectures on {G}lauber dynamics
%for discrete spin models. \newblock In P.~Bernard, editor, \emph{Lectures on
%probability theory and statistics ({S}aint-{F}lour {XXVII}, 1997)}, volume
%1717 of \emph{Lecture Notes in Mathematics}, pages 93--191, Berlin, 1999. Springer-Verlag.


%\bibitem {McK66}H.P. McKean. \newblock A class of {M}arkov processes
%associated with nonlinear parabolic equations.
%\newblock {\em Proc. Natl Acad Sci U.S.A>}, 56(6):1907--1911, 1966.


%\bibitem {Oes84}Oelschlager. \newblock A martingale approach to the law of
%large numbers for weakly interacting stochastic processes.
%\newblock {\em Ann. Probab.}, 16(1):244--294, 1984.


\bibitem {Spi71}F.~Spitzer. \newblock {\em Random fields and interacting particle systems},
Mathematical Association of America,
1971. \newblock Notes on lectures given at the 1971 MAA Summer seminar.

%\bibitem {stroock05}D.~W. Stroock. \newblock \emph{An {I}ntroduction to
%{M}arkov {P}rocesses}, volume 230 of \emph{Graduate Texts in Mathematics}.
%\newblock Springer, Berlin, 2005.


%\bibitem {Szn91}A.-S. Sznitman.
%\newblock {\em Topics in propagation of chaos}, volume 1464, pages 165--251.
%\newblock Springer, 1991.


\bibitem {tamura87}Y.~Tamura. \newblock Free energy and the convergence of
distributions of diffusion processes of {M}c{K}ean type. \newblock \emph{J.
Fac. Sci. Univ. Tokyo, Sect. IA, Math.}, 34\penalty0 (2):\penalty0 443--484, 1987.

%\bibitem {Tib10}D.~Tibi. \newblock Metastability in communication networks.
%\newblock Preprint. arXiv:1002.0796.


\bibitem {veretennikov06}A.~Y. Veretennikov. \newblock On ergodic measures for
{M}c{K}ean-{V}lasov stochastic equations. \newblock In H.~Niederreiter and
D.~Talay, editors, \emph{Monte {C}arlo and quasi-{M}onte {C}arlo methods
2004}, pages 471--486, Berlin, 2006. Springer.
\end{thebibliography}
\end{document}